\documentclass[10pt]{amsart}
\usepackage{amssymb}
\usepackage{graphicx}
\usepackage{dsfont}
\usepackage{xcolor}

\usepackage[colorlinks,citecolor=darkblue,urlcolor=blue,bookmarks=false,hypertexnames=true]{hyperref}
\definecolor{darkblue}{HTML}{004C93} 
\definecolor{MainRed}{rgb}{.6, .1, .1}
\usepackage{tzplot}

\numberwithin{equation}{section}
\newcommand{\R}{\mathbb{R}}

\newcommand{\ve}{\varepsilon}
\newcommand{\vp}{\varphi}
\newcommand{\Om}{\Omega}
\newcommand{\bal}{\begin{aligned}}
	\newcommand{\eal}{\end{aligned}}
\newcommand{\ben}{\begin{equation}}
	\newcommand{\een}{\end{equation}}
\newcommand{\be}{\begin{equation*}}
	\newcommand{\ee}{\end{equation*}}

\newcommand{\ua}{u_\alpha}

\newcommand{\ma}{\mu_\alpha}

\def \rr {\mathbb{R}}

\def \ep{\epsilon}
\def \bb{\hbox}
\def \vv {\mathbb{\Vert}}

\def \xl{x_\alpha}

\def \ls{+\infty}
\def\la{\alpha}
\def \vb{\bar{v}_{\alpha}}

\def \ma{\mu_{\alpha}}

\def \tg{\tilde{G}}
\def \yl{y_{\alpha}}
\def \tva{\tilde{v}_{\alpha}}
\def \hva{\hat{v}_{\alpha}}

\def \Ru{R_{\nu}}

\def \va{v_{\alpha}}
\def \ha{h_{\alpha}}
\def \hi{h_{\infty}}
\def \tha{\tilde{h}_{\alpha}}
\def \na {\nu_{\alpha}}

\def \vo {v_{\infty}}

\def \nn{\mathbb{N}}
\def \xin{x_{\infty}}
\def \r {\rangle}
\def \l{\langle}
\def \bsa{B_{\delta\sqrt{\ma}}}
\def \hha{\hat{h}_{\alpha}}
\def \sa {\sqrt{\ma}}

\def \ra{r_{\alpha}}

\newtheorem{step}{Step}
\def \h{\mathcal{H}}
\def \th{\widetilde{\mathcal{H}}}

\def \da{d_{\alpha}}

\def \lal{\lim_{\alpha\to +\infty}}
\def \hb{\bar{h}_{\alpha}}
\def \vbi{\bar{v}_{\infty}}

\def \ro{\rho}
\DeclareMathOperator{\loc}{loc}

\theoremstyle{plain}

\newtheorem{theo}{Theorem}[section]
\newtheorem{prop}{Proposition}[section]
\newtheorem{lemme}{Lemma}[section]
\newtheorem{corol}{Corollary}[section]
\theoremstyle{remark}
\newtheorem{rem}{Remark}[section]

\title[]{Compactness results for Sign-Changing Solutions of critical nonlinear elliptic equations of low energy }
\author{Hussein Cheikh Ali}

\author{Bruno Premoselli}
\address{Universit\'e Libre de Bruxelles, Service d'analyse, CP 218, Boulevard du Triomphe, B-1050 Bruxelles, Belgique.}
\email{hussein.cheikh-ali@ulb.be, bruno.premoselli@ulb.be}

\thanks{The second author was supported by the Fonds Th\'elam, an ARC Avanc\'e 2020 grant and an EoS FNRS grant} 
\date{December 1, 2024}
\begin{document}
	\begin{abstract}
Let $\Om$ be a bounded, smooth connected open domain in $\rr^n$ with $n\geq 3$. We investigate in this paper compactness properties for the set of sign-changing solutions $v \in H^1_0(\Om)$ of
		\begin{equation}\label{eqabs} \tag{*}
			\left\{\begin{aligned}
						-\Delta v+h v& =\left|v\right|^{2^*-2}v &\hbox{ in } \Omega, \\
			v& = 0  &\hbox{ on } \partial \Omega
		\end{aligned}\right.
	\end{equation}
where $h\in C^1(\overline{\Om})$ and $2^*:=2n/(n-2)$. Our main result establishes that the set of \emph{sign-changing} solutions of \eqref{eqabs} at the lowest sign-changing energy level is unconditionally compact in $C^2(\overline{\Om})$ when $3 \le n \le 5$, and is compact in $C^2(\overline{\Om})$ when $n \ge 7$ provided $h$ never vanishes in $\overline{\Om}$. In dimensions $n \ge 7$ our results apply when $h >0$ in $\overline{\Om}$ and thus complement the compactness result of \cite{DevillanovaSolimini}. Our proof is based on a new, global pointwise description of blowing-up sequences of solutions of \eqref{eqabs} that holds up to the boundary. We also prove more general compactness results under perturbations of $h$.
	\end{abstract}
	\maketitle
	\section{Introduction}
		
\subsection{Statement of the results}

Let $\Om\subset \rr^n$ be a smooth bounded connected open set in $\rr^n$, $n\geq 3$, $h \in C^1(\overline{\Om})$ and  $2^*:=2n/(n-2)$. In this paper we investigate solutions $v \in H^1_0(\Om)$ of 
		\begin{equation}\label{BN}
			\left\{\begin{aligned}
						-\Delta v+h v& =\left|v\right|^{2^*-2}v &\hbox{ in } \Omega, \\
			v& = 0  &\hbox{ on } \partial \Omega.
		\end{aligned}\right.
	\end{equation}
	Here and in the sequel, we let $\vv \cdot \vv_p$ be the usual norm of $L^p(\Om)$ for $1\leq p\leq \infty$, and $H_0^1(\Om)$ be the completion of $C_c^{\infty}(\Om)$ with respect to the norm
	 $$\vv v \vv_{H_0^1}^2:=\int_{\Om}|\nabla v|^2\, dx.$$
For simplicity we will assume throughout this paper that $- \Delta + h$ is coercive, that is, that there exists $C>0$ such that 
		 \begin{equation*}
		 	\int_{\Om} \left(|\nabla v|^2+h v^2 \right) \, dx \geq C \int_{\Om} |\nabla v|^2\, dx \bb{ for all } v\in H_0^1(\Om).
		 \end{equation*}
Under this assumption, the existence of positive solutions of \eqref{BN} is very well-understood. We let 
	\begin{equation}\label{Iho}
		I_{h}(\Om):=	\inf_{v\in H_0^1(\Om) \backslash \{0 \} }\frac{\int_\Om \left( |\nabla v|^2+h v^2\right)\, dx}{\left( \int_\Om |v|^{2^*}\, dx\right) ^{\frac{2}{2^*}}}.
	\end{equation}
Br\'ezis-Nirenberg \cite{BN} proved that when $n \ge 4$ positive ground states attaining \eqref{Iho} exist if and only $h<0$ somewhere in $\Om$. When $n=3$, Druet  \cite{Druetdim3} proved that positive ground states attaining \eqref{Iho} exist if only if $m_h >0$ somewhere in $\Om$, where $m_h$ is the so-called mass-function of the operator $-\Delta +h$. This function is defined as follows: let $G_h$ be the Green's function for $-\Delta +h$ with Dirichlet boundary conditions in $\Om$. Then, when $n=3$, we have
	\begin{equation*}
		G_{h}(x,y)=\frac{1}{4 \pi |x-y|}+g_{h}(x,y) \bb{ for all } y\in \Om\backslash \{x\}
	\end{equation*}
for some $g_h \in C^{0,1}(\overline{\Om}\times \overline{\Om})$, and we define $m_h(x) = g_{h}(x,x)$.
Under these assumptions, \cite{BN} and \cite{Druetdim3} also prove that we have $I_{h}(\Om) < K_n^{-2}$, where 
	\begin{equation}\label{defkn}
		K_n^{-2}:=\inf_{v\in C_c^{\infty}(\rr^n)\backslash\{0\}} \frac{\int_{\rr^n}|\nabla v|^2\, dx}{\left( \int_{\rr^n}| v|^{2^*}\, dx\right)^{\frac{2}{2^*}}} 
	\end{equation}
is the optimal constant in Sobolev's inequality in $\R^n$. An explicit expression of $K_n$ can be found in \cite{AubinYamabe, Talenti}. It is simple to see that if $v \in H^1_0(\Om)$ attains $I_h(\Om)$ then 
\begin{equation} \label{thresholdpositive}
\int_{\Om} |v|^{2^*} \, dx = I_h(\Om)^{\frac{n}{2}} < K_n^{-n}.
\end{equation}
The existence of sign-changing solutions for problem  \eqref{BN} has also attracted a lot of attention. Existence results for a general function $h \in C^1(\overline{\Om})$ are in \cite{BartschWeth}. When  $h \equiv - \lambda$, for $ \lambda \in (0, \lambda_1)$, equation \eqref{BN} is the so-called Br\'ezis-Nirenberg problem:
		\begin{equation}\label{BNlambda}
			\left\{\begin{aligned}
						-\Delta v - \lambda v& =\left|v\right|^{2^*-2}v &\hbox{ in } \Omega \\
			v& = 0  &\hbox{ on } \partial \Omega,
		\end{aligned}\right.
	\end{equation}
for which existence results have been obtained in \cite{CeramiFortunatoStruwe, CapozziFortunatoPalmieri, FortunatoJannelli, Solimini,  DevillanovaSolimini, ClappWeth, SchechterZou}. The existence of a sign-changing solution of least-energy (among all sign-changing solutions) for \eqref{BNlambda} when $\lambda \in (0, \lambda_1)$ -- the range in which $- \Delta - \lambda$ is coercive -- was proven in \cite{CeramiSoliminiStruwe} when $n \ge 6$ (see also \cite{ChenZou} for a new proof) while it was proven in \cite{RoselliWillem, TavaresYouZou} when $n=4,5$. The existence of least-energy sign-changing solutions for \eqref{BNlambda} is not yet known when $n=3$.

\medskip

In this paper we focus on compactness properties for solutions of \eqref{BN}. We let $(h_\alpha)_{\alpha \in \mathbb{N}}$ be a sequence of $C^1$ functions that converge to $h$ in $
C^1(\overline{\Om})$ and we let $(v_\alpha)_{\alpha \in \mathbb{N}}$ be a sequence of solutions in $H^1_0(\Om)$ of 
		\begin{equation}\label{BNalpha}
	\left\{\begin{aligned}
			-\Delta v_\alpha+h_\alpha v_\alpha& =\left|v_\alpha\right|^{2^*-2}v_\alpha&\hbox{ in } \Omega, \\
			v_\alpha& = 0  &\hbox{ on } \partial \Omega
		\end{aligned}\right.
	\end{equation}
	satisfying $\limsup_{\alpha \to + \infty} \Vert v_\alpha \Vert_{H^1_0} < + \infty$. We will say that $(v_\alpha)_\alpha$ is \emph{sign-changing} if, for any $\alpha$, $(v_\alpha)_+ = \max(v_\alpha,0)$ and $(v_\alpha)_- = - \min(v_\alpha,0)$ are both nonzero. We investigate under which assumptions on $h$ the sequence $(v_\alpha)_{\alpha \in \mathbb{N}}$ converges in a strong topology. Our main result answers this question when $(\va)_{\la \in \nn}$ has minimal energy: 
	\begin{theo}\label{maintheo2bis}
	Let $\Om$ be a smooth bounded connected domain of $\rr^n$, $n\geq 3$, and $(\ha)_{\alpha\in\nn}$ be a sequence that converges in $C^{1}(\overline{\Om})$ towards $h$. Assume that $-\Delta+h$ is coercive and that $I_{h}(\Om)<K_{n}^{-2}$. Let   $(\va)_{\alpha\in \nn}\in H_0^1(\Om)$ be a sequence of solutions of \eqref{BNalpha} such that
\begin{equation}\label{cond:energy1}
\limsup_{\alpha\to +\infty}\int_{\Om} |\va|^{2^*}\, dx\leq K_{n}^{-n}+I_{h}(\Om)^{\frac{n}{2}}
\end{equation}
and assume that one of the following assumptions is satisfied: 
		\begin{itemize}
		\item  either $n \in \{3,4,5\}$ and, for all $\alpha \ge 0$, $\va$ is sign-changing, or
		\item  $n \ge 7$ and $h \neq 0$ at every point in $\overline{\Om}$.
	\end{itemize}
Then, up to a subsequence, $(\va)_{\alpha\in \nn}$ strongly converge in $C^2(\overline{\Om})$ to a non-zero solution of \eqref{BN}.
\end{theo} 
Recall that $I_{h}(\Om)$ is defined in \eqref{Iho}. In the particular case where $h_\alpha \equiv h$, Theorem \ref{maintheo2bis} implies the following compactness result for solutions of \eqref{BN}: 

\begin{corol} \label{maincorol}
Let $\Om$ be a smooth bounded connected domain of $\rr^n$, $n\geq 3$, and let $h \in C^{1}(\overline{\Om})$ be such that $-\Delta+h$ is coercive and $I_{h}(\Om)<K_{n}^{-2}$.
\begin{itemize}
\item Assume that $n \in \{3,4,5\}$. There exists $\ve = \ve(n,\Om)>0$ such that the set of \emph{sign-changing} solutions $v$ of \eqref{BN} satisfying 
$$\int_\Om |v|^{2^*}\, dx \le K_n^{-n} + I_h(\Om)^{\frac{n}{2}}  + \ve $$
is precompact in the $C^2(\overline{\Om})$-topology.
\item Assume that $n \ge 7$ and $h \neq 0$ in $\overline{\Om}$. There exists $\ve = \ve(n,h,\Om) >0$ such that the set of solutions $v$ of \eqref{BN} satisfying 
$$\int_\Om |v|^{2^*}\, dx \le K_n^{-n} + I_h(\Om)^{\frac{n}{2}} + \ve $$
is precompact in the $C^2(\overline{\Om})$-topology.
\end{itemize}
\end{corol}
The energy bound \eqref{cond:energy1} is very natural when investigating sign-changing solutions of \eqref{BN}. Solutions of \eqref{BNalpha} satisfying \eqref{cond:energy1} exist: the least-energy sign-changing solutions of \eqref{BNlambda} constructed in \cite{CeramiSoliminiStruwe, TavaresYouZou}, for instance, satisfy $\int_{\Om} |v|^{2^*} \, dx < K_{n}^{-n}+I_{-\lambda}(\Om)^{\frac{n}{2}}$.  A simple application of Struwe's \cite{Struwe} celebrated compactness result (see also \cite[Lemma 3.1]{CeramiSoliminiStruwe}) shows that if a sequence $(v_\alpha)_{\alpha \in \mathbb{N}}$ of solutions of \eqref{BNalpha} \emph{changes sign} and satisfies $\lim_{\alpha \to + \infty} \Vert v_\alpha \Vert_\infty = + \infty$ (we will say in this case that $(\va)_{\alpha \in \mathbb{N}}$ \emph{blows-up}), then 
$$ \int_{\Om}|\va|^{2^*}\, dx\geq K_{n}^{-n}+I_{h}(\Om)^{\frac{n}{2}} + o(1) $$
as $\alpha \to + \infty$. The threshold $K_{n}^{-n}+I_{h}(\Om)^{\frac{n}{2}}$ is therefore the direct counterpart, for sign-changing solutions, of the minimal energy threshold $K_n^{-n}$ that ensures the existence of positive ground state solutions in \eqref{thresholdpositive}. In this respect, Theorem \ref{maintheo2bis} and Corollary \ref{maincorol} have to be understood as the first compactness result for \eqref{BNalpha}, at the lowest energy-level for sign-changing blow-up, when $I_{h}(\Om)$ is attained.

\medskip

Theorem \ref{maintheo2bis} shows that when $3 \le n \le 5$ \emph{sign-changing} solutions are unconditionally compact in $C^2(\overline{\Om})$ under assumption \eqref{cond:energy1}. By contrast, without further assumptions on $h$, the set of \emph{positive} solutions satisfying \eqref{cond:energy1} is not compact in general when $3 \le n \le 5$. For equation \eqref{BNlambda}, for instance, families of positive solutions whose energy converges to $K_n^{-n}$ and which are not compact in $C^2(\overline{\Om})$  have been constructed in \cite{MussoPistoia2, Rey} when $n \ge 4$ and $\lambda \to 0+$, and in \cite{DelPinoDolbeaultMusso} when $n=3$ and $\lambda \to \lambda_*$ from above, where  $\lambda_*$ satisfies $\max_\Om m_{\lambda_*} = 0$. When $3 \le n \le 5$, Theorem \ref{maintheo2bis} is therefore unexpected since sign-changing solutions of equations like \eqref{BNalpha} are known to exhibit a much richer and more erratic behavior than positive ones. When $n \ge 7$, Theorem \ref{maintheo2bis} applies to positive and sign-changing sequences of solutions $(\va)_{\alpha \in \mathbb{N}}$ and Corollary \ref{maincorol} generalises the well-known compactness theorem for energy-bounded solutions of \eqref{BNlambda} proven in \cite{DevillanovaSolimini}. It is still an open question to know whether Theorem \ref{maintheo2bis} holds true for any energy-bounded sequence $(\va)_{\la \in \nn}$ without the assumption \eqref{cond:energy1} when $n \ge 7$ and $h \neq 0$ in $\overline{\Om}$.

\medskip

Dimension $6$ is excluded from Theorem \ref{maintheo2bis}. In this case we prove: 
\begin{prop} \label{prop:cas6}
Let $\Om$ be a smooth bounded domain of $\rr^6$ and $(\ha)_{\alpha\in\nn}$ be a sequence that converges in $C^{1}(\overline{\Om})$ towards $h$. Assume that $-\Delta+h$ is coercive and that $I_{h}(\Om)<K_{6}^{-2}$. Let   $(\va)_{\alpha\in \nn}\in H_0^1(\Om)$ be any sequence of solutions of \eqref{BNalpha} satisfying \eqref{cond:energy1} and assume that $\Vert v_\alpha \Vert_{\infty} \to +\infty$ as $ \alpha \to + \infty$. Then there exists $v_\infty \in H^1_0(\Om)$, $v_\infty >0$ in $\Om$, attaining $I_{h}(\Om)$ such that $v_\alpha$ converges weakly but not strongly to $ \pm v_\infty$ in $H^1_0(\Om)$ and there exists $x_\infty \in \Om$ such that 
$$h(x_\infty) = \pm   2v_\infty(x_\infty).$$ 
\end{prop}
Compactness of \emph{sign-changing} solutions of \eqref{BNalpha} satisfying \eqref{cond:energy1} does not hold when $n=6$: in \cite{PistoiaVairadim6}, for instance, the authors constructed a non-compact family $(v_\lambda)_\lambda$ of sign-changing solutions of \eqref{BNlambda} which blows-up as $\lambda$ converges to some $\lambda_0 >0$ that satisfies $\lambda_0 = 2 \Vert v_{0} \Vert_{\infty}$, where $v_0$ attains $I_{- \lambda_0}(\Om)$ (the existence of such $(\lambda_0, v_0)$ is also proven in \cite{PistoiaVairadim6}). This six-dimensional phenomenon has been known for a while for positive solutions, where it was first highlighted in \cite{DruetYlowdim}.

\subsection{Strategy of proof and outline of the paper}

For \emph{positive} solutions there is a vast literature addressing the issue of compactness of equations like \eqref{BNalpha}  through blow-up analysis. On open sets of $\R^n$ with Dirichlet boundary conditions we mention for instance \cite{Druetdim3, DruetLaurain,LaurainKonig, LaurainKonig2} for \eqref{BN}, \cite{druet2012lin} for Lin-Ni type problems with Neumann boundary conditions and  \cite{GhoussoubMazumdarRobert} for singular Hardy-Sobolev type problems. On closed manifolds we mention \cite{DruetJDG} for compactness of energy-bounded solutions and  the series of works related to the compactness of the Yamabe equation: \cite{LiZhu, DruetJDG, Marques, KhuMaSc} (see also \cite{HebeyZLAM} for additional references). On manifolds with boundary we refer to \cite{mesmar2024concentration}. For \emph{sign-changing} solutions of critical elliptic equations on closed manifolds, compactness results have been recently obtained: we refer for instance to \cite{PremoselliVetois, PremoselliVetois2, PremoselliVetois3, PremoselliVetois4, PremoselliRobert}.  Concerning problem \eqref{BNlambda} in particular, there is a vast literature on the construction and the behavior of blowing-up solutions:  we mention for instance \cite{BAEMP1, BAEMP2, Druetdim3, DruetLaurain, LaurainKonig, LaurainKonig2, IacopettiPacella, IacopettiVaira2, MussoPistoia2, musso2024nodal, Premoselli12, vaira2015new} and the references therein.

\medskip

Our approach in this paper is strongly inspired from these references. We proceed by contradiction: under the assumptions (and with the notations) of Theorem \ref{maintheo2bis}, and by \cite{Struwe}, if $(v_\alpha)_{\alpha \in \mathbb{N}}$ does not strongly converge in $H^1_0(\Om)$ we have, up to a subsequence,
\begin{equation} \label{sumvlaintro}
	\va=B_{\alpha} \pm \vo+ o(1)  \bb{ in } H_0^1(\Om)
\end{equation}
as $\alpha \to +\infty$, where $\vo  \ge 0 $ solves \eqref{BN} and where $B_\alpha$ is a positive bubbling profile that concentrates at some point $x_\alpha \in \Om$ and is modeled on a positive solution of $-\Delta B = B^{2^*-1}$ in $\R^n$ (see \eqref{sumvla} below for more details). We perform an asymptotic analysis of $\va$ near $x_\alpha$ at different scales and obtain necessary conditions on $h$ for blow-up to occur. The contradiction follows from these conditions: to prove Theorem \ref{maintheo2bis} when $3 \le n \le 5$, for instance, we prove that if \eqref{sumvlaintro} holds we simultaneously have  $\vo \equiv 0$ and $\vo >0$ in $\Om$.  In order to investigate the behavior of $\va$ near $x_\alpha$ we prove in this paper new pointwise estimate on $\va$, up to the boundary, that improve \eqref{sumvlaintro} in strong spaces. We precisely prove that 
\begin{equation} \label{theorieC0intro}
 \left \Vert \frac{\va - \Pi B_\alpha \mp \vo}{B_\alpha + \vo} \right \Vert_{\infty} \to 0 
\end{equation}
as $\alpha \to + \infty$, where $\Pi B_\alpha$ is the projection of $B_\alpha$ in $H^1_0(\Om)$ defined by \eqref{projbulle} below (see Theorem \ref{maintheo1} below for a precise statement). Estimate \eqref{theorieC0intro} provides an accurate control on $\va$ up to $\partial \Om$ and is particularly useful close to $\partial \Om$, where, at first order,  $\Pi B_\alpha $ deviates from $B_\alpha$  and $v_\infty$ vanishes. To the best of our knowledge this is the first time that a similar estimate is proven. We heavily rely on estimate \eqref{theorieC0intro} to rule out the possibility that the concentration point $\xl$ converges to a point in $\partial \Om$: this is both the main difficulty that we face in the proof of Theorem  \ref{maintheo2bis} and the main novelty of our analysis, and is deeply related to the sign-changing nature of the solutions we consider (see Remarks \ref{concentration:au:bord} and \ref{concentration:au:bord2} below for a detailed explanation of this fact). 

\medskip

The structure of the paper is as follows. In Section \ref{secprovetheo1} we prove Theorem \ref{maintheo1} and establish \eqref{theorieC0intro}. In Section \ref{secpohozaev} we apply it to obtain necessary conditions for the blow-up of $(v_\alpha)_{\alpha \in \mathbb{N}}$ by means of suitable Pohozaev identities at different scales. We separately treat the interior blow-up case (Proposition \ref{prop:blowup:interieur}) and the boundary blow-up case (Propositions \ref{prop:blowup:bord}, \ref{prop:blowup:bord2} and \ref{prop:blowup:bord3}), and we deduce our main result, Theorem \ref{maintheo2bis}, from this analysis. Finally, Appendix \ref{annexe} contains the proof of a few technical results that are used throughout Section \ref{secpohozaev}.

\section{The $C^0$-theory for blow-up}\label{secprovetheo1}

In this section we let $h_{\infty}\in C^{0}(\overline{\Om})$ and consider a family of functions $(\ha)_{\alpha\in \nn}\in C^{1}(\overline{\Om})$ such that 
\begin{equation}\label{haconvhi}
	\lim_{\alpha\to +\infty}\ha =\hi \bb{ in } C^{0}(\overline{\Om}). 
\end{equation} 
We assume that $-\Delta+h_{\infty}$ is coercive in $H^1_0(\Om)$ and that $I_{\hi}(\Om)<K_n^{-2}$, where $I_{\hi}(\Om)$ is as in \eqref{Iho}, so that positive ground states of \eqref{BN} with $h = h_\infty$ exist. We consider a sequence of functions $(\va)_{\alpha\in \nn}$ in $H_0^1(\Om)$ such that, for all $\alpha\in \nn$, $v_{\alpha}$ is a solution to
\begin{equation}\label{critvlambda}
	\left\{\begin{aligned}
		-\Delta \va +h_{\alpha} \va &=\left|\va\right|^{2^*-2}\va \hbox{ in } \Omega, \\
		\va&= 0  \hbox{ in } \partial \Omega.
	\end{aligned}\right.
\end{equation}
We assume that
\begin{equation}\label{limvlaL2star}
	\limsup_{\alpha\to +\infty}\int_{\Om} |\va|^{2^*}\, dx \le K_n^{-n} +I_{\hi}(\Om)^{\frac{n}{2}}. 
\end{equation} 
We also assume that $(\va)_{\alpha \in \nn}$ blows-up, that is
\begin{equation}\label{limvlaLinfini}
	\lim_{\alpha\to +\infty} \vv \va\vv_{\infty}=+\infty.
\end{equation} 
By \eqref{limvlaL2star} and \eqref{limvlaLinfini}, and following \cite{Struwe} (see also \cite{StruweVariationalMethods}), we get that, up to a subsequence
\begin{equation}\label{sumvla}
	\va=B_{\alpha}\pm\vo+\vp_{\alpha} \bb{ in } H_0^1(\Om), 
\end{equation}
where $\vv \vp_{\alpha}\vv_{H_0^1}\to 0$ as $\alpha \to  +\infty$. In \eqref{sumvla} $v_\infty$ is a solution of \eqref{BN} with $h = h_\infty$ and we have let 
\begin{equation}\label{Bla}
	B_{\alpha}(x):=\ma^{-\frac{n-2}{2}}B_0(\ma^{-1}(x-\xl))
	\quad \text{ for } x \in \Om,
\end{equation}
where $(\xl)_{\alpha \in \nn}$ and $(\ma)_{\alpha \in \nn}$ are respectively sequences of points in $\Om$ and positive real numbers, and where we have let
\begin{equation}\label{B0}
B_0(x)=\left( 1+\frac{|x|^2}{n(n-2)}\right) ^{1-\frac{n}{2}} \bb{ for any } x\in \rr^n.
\end{equation} 
It is well-known that $B_0$ satisfies $-\Delta B_0 = B_0^{2^*-1}$ in $\R^n$ and achieves $K_n^{-2}$ in \eqref{defkn}. As a consequence of \eqref{sumvla}, we have 
\begin{equation}\label{limitefaiblevla}
	\lim_{\alpha\to +\infty}\va= \pm \vo \bb{ weakly in } H_0^1(\Om)
\end{equation}
and 
$$\lim_{\alpha \to +\infty} \int_{\Om} |\va|^{2^*} \, dx =  K_n^{-n} + \int_{\Om} |\vo|^{2^*} \, dx. $$
A consequence of \eqref{limvlaL2star} and of the assumption $I_{\hi}(\Om)<K_n^{-2}$ is that either $v_\infty \equiv 0$ or $\vo$ is a least-energy positive solution of
\begin{equation}\label{critu0}
	\left\{\begin{aligned}
		-\Delta \vo +\hi \vo & =\vo^{2^*-1}\hbox{ in } \Omega, \\
		\vo&>0  \bb{ in }\Om,\\
		\vo&= 0  \hbox{ on } \partial \Omega.
	\end{aligned}\right.
\end{equation} 
If $\va$ is assumed to change sign for all $\alpha \ge 1$, that is if $(\va)_+$ and $(\va)_- $ are nonzero, the arguments in \cite[Lemma 3.1]{CeramiSoliminiStruwe} show that $\vo >0$, and hence that 
$$\lim_{\alpha \to +\infty} \int_{\Om} |\va|^{2^*} \, dx =  K_n^{-n} + I_{h_\infty}(\Om)^{\frac{n}{2}}. $$
This observation will be important in the proof of Theorem \ref{maintheo2bis} but will not be used in this Section. Without loss of generality we can assume that $(\xl)_{\alpha \in \nn}$ and $(\ma)_{\alpha \in \nn}$ are chosen as follows:
\begin{equation}\label{defma}
	\big| \va(\xl)\big| =\left\| \va(x)\right\|_{\infty} \bb{ and }\ma:=\big| \va(\xl)\big|^{-\frac{2}{n-2}},
\end{equation}
so that $\xl \in \Omega$. Note that \eqref{limvlaLinfini} implies that $\ma\to 0$ as $\alpha\to+\infty$. We will denote by $x_{\infty}\in \overline{\Om}$ the limit of the $x_{\alpha}$'s as $\alpha\to +\infty$. In the case where $v_\infty >0$, Hopf's lemma shows that there exists $C_0>0$ such that 
\begin{equation}\label{estu0}
	C_0^{-1}\,d(x,\partial\Om)\leq \vo(x)\leq C_0\, d(x,\partial\Om) \bb{ for all } x\in \Om,
\end{equation} 
where $d(x,\partial\Om):=\inf\{|x-y|: y\in\partial \Om\}$ is the distance of $x$ to boundary. In \eqref{sumvla} we used the notation $\va=B_{\alpha}\pm\vo+\vp_{\alpha}$, which classically means either $\va = B_\alpha + \vo + \vp_{\alpha}$ or $\va = B_\alpha - \vo + \vp_{\alpha}$. It will often be more convenient to substract $B_\alpha \pm \vo$ to $\ua$ (for instance in the statement of Theorem \ref{maintheo1} below), which we will thus write as 
$$ \va - B_\alpha \mp \vo = \vp_\alpha$$
so that the sign convention is satisfied.

\medskip

The purpose of this section is to turn \eqref{sumvla} into a decomposition in strong spaces, and to obtain sharp pointwise estimates on $\va$. In order to state our main result we need to introduce a few more notations. For $\alpha$ large, thanks to \eqref{haconvhi}, $-\Delta+\ha$ is coercive in $H^1_0(\Om)$. We can thus let $G_{\alpha}$ be the Green's function of $-\Delta+h_{\alpha}$ in $\Omega$ with Dirichlet boundary conditions. By standard properties of the Green's function (see \cite{RobDirichlet}), there exists $C>0$ such that for all $\alpha \ge 1$ we have 
\begin{equation}\label{estGlai}
	G_{\alpha}(y,x)\leq \frac{C}{|y-x|^{n-2}} \min\left\lbrace 1,\frac{d(y,\partial\Om)d(x,\partial\Om)}{|y-x|^2}\right\rbrace \bb{ for all }\, x,y \in \Om,\,  x\neq y,
\end{equation}
and 
\begin{equation}\label{estnablaGlai}
	\big|\nabla G_{\alpha}(y,x)\big|\leq C |y-x|^{1-n} \bb{ for all } \, x,y \in \Om,\, x\neq y.
\end{equation}
For $\alpha \ge1$,  we let $\Pi B_\alpha$ be the unique solution in $H^1_0(\Om)$ of 
\begin{equation} \label{projbulle}
\left \{ \begin{aligned}
\big(-\Delta + h_\alpha\big) \Pi B_\alpha & = B_{\alpha}^{2^*-1} \quad \text{ in } \Om \\
\Pi B_\alpha & = 0 \quad \text{ on } \partial \Om.
\end{aligned} \right. 
\end{equation}
Since $B_\alpha$ satisfies $-\Delta B_\alpha = B_\alpha^{2^*-1}$ in $\R^n$ by \eqref{Bla} and \eqref{B0} we easily see with \eqref{projbulle} that $B_\alpha - \Pi B_\alpha \to 0$ in $H^1_0(\Om)$ as $\alpha \to + \infty$. Thus \eqref{sumvla} rewrites as 
 \begin{equation} \label{sumvla2}
 	\va=\Pi B_{\alpha}\pm\vo+o(1)  \bb{ in } H_0^1(\Om) \text{ as } \alpha \to + \infty.  
\end{equation}
A representation formula for $\Pi B_\alpha$ together with \eqref{estGlai} shows that there exists $C>0$ such that for all $x \in \Om$ and all $\alpha \ge 1$ we have 
\begin{equation} \label{contprojbulle}
0 < \Pi B_\alpha(x) \le C B_{\alpha}(x),
\end{equation}
where positivity follows from the coercivity of $-\Delta + h_\alpha$. We can now state the main result of this Section: 
\begin{theo}\label{maintheo1}
	Let $\Om$ be a smooth bounded domain of $\rr^n$, $n\geq 3$, and $(\ha)_{\alpha\in\nn}$ be a sequence of functions that converges in $C^{0}(\overline{\Om})$ to $\hi$. We assume that $-\Delta+\hi$ is coercive in $H^1_0(\Om)$ and that $I_{\hi}(\Om)<K_{n}^{-2}$. Let $(\va)_{\alpha\in\nn}\in H_0^1(\Om)$ be a sequence of solutions of \eqref{critvlambda} that satisfies \eqref{limvlaL2star}, \eqref{limvlaLinfini} and \eqref{sumvla}. 
	There exists a sequence $(\ve_\alpha)_{\alpha \in\nn}$ of positive real numbers converging to $0$ such that, up to a subsequence we have, for any $x \in \Om$ and $\alpha \ge 1$,
	\begin{equation}\label{theorieC0}
		\Big| v_{\la}(x) - \Pi B_\alpha(x) \mp v_\infty(x) \Big| \leq \ve_\alpha \big(  B_{\alpha}(x)+\vo(x)\big). 
		\end{equation}
\end{theo}
Pointwise descriptions of blowing-up solutions as in Theorem \ref{maintheo1} were first obtained for \emph{positive} solutions of critical Schr\"odinger-type equations on manifolds without boundary: see for instance \cite{DHpde, DHR} (see also \cite{HebeyZLAM}). For \emph{positive} solutions of equations like \eqref{critvlambda} in bounded open subsets of $\R^n$ they were recently obtained in \cite{LaurainKonig, LaurainKonig2}. Similar estimates have been obtained for positive solutions of Hardy-Sobolev equations in  \cite{cheikh2022second, GhoussoubMazumdarRobert}. These sharp pointwise estimates have proven crucial in order to obtain compactness and stability results for critical stationary elliptic equations \cite{DruetJDG, DruetLaurain}. When it comes to \emph{sign-changing} blowing-up solutions, a general pointwise description as in Theorem \ref{maintheo1}, on manifolds without boundary, has been recently obtained in \cite{premoselli2024priori, PremoselliRobert}, and subsequent compactness results have been proven in \cite{PremoselliRobert, PremoselliVetois2, PremoselliVetois3}. Theorem \ref{maintheo1} is, to our knowledge, the first instance where sharp pointwise estimates for blowing-up solutions of equations like \eqref{critvlambda} are obtained up to the boundary of $\Om$. Note indeed that in Theorem \ref{maintheo1} we do not assume that the concentration point $x_\infty = \lim_{\alpha \to + \infty} x_\alpha$ is an interior point in $\Om$. It may happen that $x_\infty \in \partial \Om$: the real novelty of Theorem \ref{maintheo1} is that \eqref{theorieC0} holds regardless of the speed of convergence of $x_\alpha$ to $\partial \Om$, uniformly in $x \in \overline{ \Om}$. This creates additional technical difficulties that we overcome in the course of the proof. 
 
 \medskip
 
We prove Theorem \ref{maintheo1} by taking inspiration from the arguments in \cite{DHpde} (see also \cite{HebeyZLAM}). Throughout this section we let $\Om$ be a smooth bounded domain in $\rr^n$, $n\geq 3$, $(\ha)_{\alpha \in \nn} \in C^{0}(\overline{\Om})$ and  $(\va)_{\alpha \in \nn}\in H_0^1(\Om)$  be such that \eqref{haconvhi}, \eqref{critvlambda}, \eqref{limvlaLinfini},   and \eqref{sumvla} hold, and we let $(\xl)_{\alpha \in \nn} \in\Om$ and $(\ma)_{\alpha \in \nn}$  be as defined as in \eqref{defma}. We start with the following simple proposition:
	\begin{prop}\label{prop:Blowup1}
We have 
\begin{equation}\label{dasurma}
	\lim_{\la\to \ls} \frac{d(\xl,\partial\Om)}{\ma}= +\infty.
\end{equation} 
 We define the rescaled function 
	\begin{equation}\label{deftva}
		\tva(x):=\ma^{\frac{n-2}{2}}\va(\xl+\ma x) \bb{ for all } x\in \Om_{\la},
	\end{equation}
where $\Om_{\la}:= \{x\in \rr^n \bb{ such that } \xl+\ma x\in \Om\}$. Then
 \begin{equation}\label{limtvaB0}
 	\lim_{\la\to\ls} \tva(x)=B_0(x) \bb{ in }C^2_{loc}(\rr^n),
 \end{equation}
 where $B_0$ is defined in \eqref{B0}.
 \end{prop}

\begin{proof}
First, \eqref{dasurma} follows from Struwe's original result \cite{Struwe} (see also \cite[Theorem 1.2]{Mazumdar2}).  We now prove \eqref{limtvaB0}.	For $x\in \Om_{\la}:= \{x\in \rr^n \bb{ s.t. } \xl+\ma x\in \Om\}$, it is clear by \eqref{critvlambda} and \eqref{deftva} that
	\begin{equation*}
		\left\{\begin{array}{ll}
			-\Delta \tva +\tha\ma^2 \tva=\left|\tva\right|^{2^*-2}\tva &\hbox{ in } \Om_{\la}, \\
			\tva=0  &\hbox{ on } \partial \Om_{\la},
		\end{array}\right.
	\end{equation*}
where $\tha(x)=\ha(\xl+\ma x)$ and $\tva$ is defined in \eqref{deftva}. We remark that $|\tva|\leq |\tva(0)|= 1$. It follows from \eqref{haconvhi} and from standard elliptic theory that, after passing to a subsequence, $\tva \to \tilde{v}$ in $C^2_{loc}(\rr^n)$,	where $\tilde{v}\in C^2(\rr^n)$ is such that 
	\begin{equation*}
		-\Delta \tilde{v}=\left|\tilde{v}\right|^{2^*-2}\tilde{v} \hbox{ in } \rr^n, 
	\end{equation*}
	and $|\tilde{v}| \leq 1$. Let $K\subset\subset \rr^n$ be a nonempty compact subset of $\rr^n$. By \eqref{sumvla} we have 
	$\tva\to B_0$ in $L^{2^*}(K)$ as $\la\to \ls$, so that $\tilde{v}=B_0$ in $K$, which proves \eqref{limtvaB0}.
\end{proof}
Using \eqref{dasurma} and standard elliptic theory, together with \eqref{projbulle} and \eqref{contprojbulle}, we also obtain that 
\begin{equation} \label{contprojbulle2}
\ma^{\frac{n-2}{2}} \Pi B_\alpha(x_\alpha + \ma x) \to B_0(x) \quad \text{ in } C^2_{\loc}(\R^n)
\end{equation}
as $\alpha \to + \infty$.The following result establishes a first pointwise control on $v_\alpha$: 

\begin{prop}\label{limiteinLinfini}
For $x \in \Om$ we let $D_{\la}(x):= |x-\xl|+\ma$. Then 
\begin{equation}\label{eq:limiteinLinfini}
D_{\la}(x)^{\frac{n-2}{2}}\Big| v_{\la} -\Pi B_{\la} \mp\vo\Big|\to 0 \bb{ in } C^0(\overline{\Om}) \bb{ as }\la\to \ls
\end{equation}
where $v_{\infty}$ and $\Pi B_{\alpha}$ are as defined in \eqref{limitefaiblevla}, \eqref{critu0} and   \eqref{projbulle}. 
\end{prop}

To prove Proposition \ref{limiteinLinfini} we proceed by contradiction: we assume that there exist $\ep_0>0$, and $(y_\la)_{\la\in \nn} \in \overline{\Om}$ such that 
\begin{eqnarray}\label{keycontra}
&&D_{\la}(\yl)^{\frac{n-2}{2}}\Big| v_{\la}(\yl)\mp\vo(\yl)-\Pi B_{\la}(\yl)\Big| \nonumber\\
&=&\max_{x\in\Om} \left(D_{\la}(x)^{\frac{n-2}{2}}\Big|v_{\la}(x)\mp\vo(x)-\Pi B_{\la}(x)\Big| \right) \geq \ep _0,
\end{eqnarray}
and, we let $(\nu_{\la})_{\alpha \in \nn}\in(0,+\infty)$ be such that 
\begin{equation}\label{defnula}
\big|v_{\la}(\yl)\big|=\nu_\la^{\frac{2-n}{2}} \quad \text{ for all } \alpha \ge 1.
\end{equation}
Since $v_\alpha, \Pi B_\alpha $ and $v_\infty$ vanish in $\partial \Om$ a first simple observation is that $y_\alpha \in \Om$.

\begin{step}\label{limDlaBla}
	We claim that 
		\begin{equation*}
		D_{\la}(\yl)^{\frac{n-2}{2}}B_{\la}(\yl)\to 0 \bb{ as } \la \to \ls.
	\end{equation*}
As a consequence, with \eqref{contprojbulle} we have 
	\begin{equation}\label{eq:limDlaBla}
		D_{\la}(\yl)^{\frac{n-2}{2}}\Pi B_{\la}(\yl)\to 0 \bb{ as } \la \to \ls.
	\end{equation}
\end{step}

\begin{proof}
Indeed, suppose on the contrary that there exists $\rho_0>0$ such that 
\begin{equation*}
D_\la(\yl)^{\frac{n-2}{2}}B_\la(\yl)\geq \rho_0,
\end{equation*}
for all $\alpha$ large enough. Hence, we have that
\begin{equation*}
1 + \frac{|\xl - \yl|}{\ma} = \frac{D_\la(\yl)}{\ma}\geq \rho_0^{\frac{2}{n-2}}\left(  1+ \frac{|\yl-\xl|^2}{\ma^2}\right).
\end{equation*}
Up to passing to a subsequence we may then assume that there exists $R>0$ such that $\lim_{\la\to \ls}\ma^{-1}|\yl-\xl|=R$. This means that 
\begin{equation}\label{DlaOma}
	D_{\la}(\yl)=O(\ma).
\end{equation}
It follows from \eqref{contprojbulle2} and \eqref{limtvaB0} that 
\begin{equation*} 
\lim_{\la\to\ls} \ma^{\frac{n-2}{2}}\Big|v_\la(\yl)-\Pi B_\la(\yl)\Big|=0.
\end{equation*}
With \eqref{DlaOma}  we thus get that 
\begin{equation*}
\lim_{\la\to\ls} D_\la(\yl)^{\frac{n-2}{2}}\Big|v_{\la}(\yl)\mp\vo(\yl)-\Pi B_{\la}(\yl)\Big|=0
\end{equation*}
which contradicts \eqref{keycontra}. 
\end{proof}

\begin{step}\label{limnula}
	We claim that 
	\begin{equation}\label{eq:limnula}
		\nu_\la \to 0  \bb{ as } \la \to \ls,
	\end{equation}
	where $\nu_{\la}$ is defined in \eqref{defnula}.
\end{step}
\begin{proof} Indeed, it follows from \eqref{keycontra} and \eqref{eq:limDlaBla} that 
\begin{equation}\label{ineqe01}
\ep_0\leq D_\la(\yl)^{\frac{n-2}{2}}\Big( \big|v_\la(\yl)\big|+\vv \vo \vv_{\infty}\Big)+o(1)
\end{equation}
as $\la \to \ls$. If $D_\la(\yl)\to 0$ as $\la \to \ls$, then \eqref{eq:limnula} follows from \eqref{ineqe01}. Suppose on the contrary that, up to a subsequence, $D_\la(\yl)\to c_0$ as $\alpha\to +\infty$ for some $c_0 >0$. It follows from \eqref{keycontra} and \eqref{eq:limDlaBla} that 
\begin{equation} \label{CVfaiblestep1}
\big|  \va(x)\mp\vo(x)\big|+o(1)\leq 2^n\,   \big| v_\la(\yl)\mp\vo(\yl)\big|+o(1),
\end{equation}
for $x\in B_{\frac{c_0}{2}}(\yl)\cap \overline{\Om}$ and all $\alpha$ sufficiently large. If $v_\la(\yl)\to +\infty$ as $\la \to \ls$, it is clear, by the definition of $\nu_\la$, that we obtain \eqref{eq:limnula}. If $v_\la(\yl)=O(1)$ standard elliptic theory together with \eqref{limitefaiblevla} and \eqref{CVfaiblestep1} proves that $v_\la \mp v_{\infty}\to 0$ in $C^2_{loc}(B_{\frac{c_0}{4}}(\yl))$ as $\la\to \ls$. This contradicts \eqref{keycontra} using \eqref{eq:limDlaBla}. We thus get that \eqref{eq:limnula} holds true. \end{proof}
For any $x\in \Om_\la:=\{ x\in\rr^n, \, \yl+\nu_\la x \in\Om\}$, we set $$w_\la(x)=\nu_\la^{\frac{n-2}{2}}v_\la(\yl+\nu_\la x).$$
By \eqref{critvlambda}, $w_\alpha$ satisfies
\begin{equation}\label{critwla}
\left\{\begin{array}{ll}
-\Delta w_\la +h_{\alpha}(\yl+\na x)\nu_\la^2 w_\la=|w_\la|^{2^*-2} w_\la&\hbox{ in } \Om_{\la}, \\
w_\la=0  &\hbox{ on } \partial\Om_{\la}.
\end{array}\right.
\end{equation}
Thanks to \eqref{defnula}, we have that $\big|w_\la(0)\big|=1$. We define a set $S$ as follows: 
$$ \begin{aligned} 
& \bullet \text{ if }  |\yl-\xl|=O(\nu_\la) \bb{ and } \ma=o(\nu_{\la}), \quad S = \Big \{  \lim_{\la\to\ls}\frac{\yl-\xl}{\nu_\la} \Big \} \\
& \bullet \text{ otherwise } S = \emptyset , 
& \end{aligned}
  $$
  where it is intended that the limit exists up to passing to a subsequence. Let us fix $K\subset\subset\rr^n \backslash S$ a compact set.
\begin{step} \label{step:limnuBla}
As $\la\to \ls$ we have 
	\begin{equation}\label{limnuBla}
		\nu_\la^{\frac{n-2}{2}}B_{\la}(\yl-\nu_\la x)\to 0 \bb{ for all } x\in K.
	\end{equation}
\end{step}
\begin{proof} Let $x \in K$. If $\nu_{\la}=o(\ma)$ then \eqref{limnuBla} is true since $B_\alpha (x) \le \ma^{- \frac{n-2}{2}}$ for any $x \in \overline{\Om}$. We now assume that $\ma=o(\nu_{\la})$: since $x\in K$, we get that $\nu_{\la}=O\left( |\yl-\xl-\nu_{\la}x|\right)$. Thus, once again \eqref{limnuBla}, holds true by definition of $B_\alpha$. We may thus assume that there exists $C>0$ such that
\begin{equation}\label{nulaeqma}
	C^{-1}\nu_{\la}\leq \ma \leq C\nu_{\la} \bb{ for all } \alpha.
\end{equation}
Assume first that $|\yl-\xl-\nu_{\la}x|=O(\ma)$. Thus, since $x\in K$ and by \eqref{nulaeqma}, we get  $|\yl-\xl|=O(\ma)$. Arguing as in the proof of Step $1$ we get a contradiction. Thus, for all $x\in K$ we have 
\begin{equation*}
	\lim_{\la\to \ls} \frac{|\yl-\xl-\nu_{\la}x|}{\ma}=+\infty.
\end{equation*}
Together with \eqref{nulaeqma} this implies that \eqref{limnuBla} holds true. 
\end{proof}

\begin{step}\label{wlabound}
	We claim that 
	\begin{equation}\label{eq:wlabound}
		w_{\la}(x)=O(1) \bb{ for all } x\in K\cap \Om_{\la}.
	\end{equation}
\end{step}
\begin{proof}
Indeed, using \eqref{keycontra} and \eqref{eq:limDlaBla} together with \eqref{limnuBla} yields
\begin{equation}\label{vlabound1}
\begin{aligned}
\left( \frac{D_{\la}(\yl+\nu_\la x)}{D_{\la}(\yl)}\right) ^{\frac{n-2}{2}}\Big|w_\la(x)\mp\nu_\la^{\frac{n-2}{2}}\vo(\yl+\nu_\la x)-\nu_\la^{\frac{n-2}{2}}\Pi B_{\la}(\yl+\nu_\la x)\Big| \\
\leq 1+o(1), \end{aligned} 
\end{equation}
 for all $x\in K\cap \Om_{\la}$. It then follows from \eqref{contprojbulle}, \eqref{eq:limnula}, \eqref{limnuBla} and \eqref{vlabound1} that 
\begin{equation}\label{vlabound2}
\left( 	\frac{D_{\la}(\yl+\nu_{\la}x)}{D_{\la}(\yl)}\right) ^{\frac{n-2}{2}}	\Big( \big|w_\la(x)\big| + o(1) \Big)\leq  1+ o(1)\bb{ for all } x\in K\cap \Om_{\la}.
\end{equation}
We claim that there exists $\eta_K >0$ such that 
$$\lim_{\la\to \ls}D_{\la}(\yl+\nu_{\la}x)D_{\la}(\yl)^{-1} \ge \eta_K  $$
 for all $x \in K \cap \Omega_\alpha$. Together with \eqref{vlabound2} this will prove that $w_\la$ is bounded in $K \cap \Omega_\alpha$. Suppose on the contrary that for a sequence $(z_\alpha)_{\alpha \in \nn}$ in $K \cap \Omega_\alpha$ we have
\begin{equation*}
|\yl-\xl+\nu_{\la}z_\alpha|+\ma=o(|\yl-\xl|)+o(\ma).
\end{equation*}
Then $|\yl-\xl|=O(\nu_{\la})$, $\ma=o(\nu_{\la})$ and 
\begin{equation*}
\lim_{\la\to \ls} \left| \frac{\yl-\xl}{\nu_{\la}}-z_\alpha\right| =0
\end{equation*}
which is a contradiction since $\liminf_{\alpha \to + \infty} d(z_\alpha, S) > 0$. 
\end{proof}

We now conclude the proof of Proposition \ref{limiteinLinfini}.

\begin{proof}[Proof of Proposition \ref{limiteinLinfini}]

We first claim that $0\in \Om_{\la}\backslash S$. If $S =\emptyset$ this is obvious. Assume thus that $S \neq \emptyset$, which implies that $|\yl-\xl|=O(\nu_\la)$ and  $\ma=o(\nu_{\la})$ as $\alpha \to + \infty$. Then, since $\nu_\la\to 0$ as $\la\to \ls$ and by \eqref{ineqe01}, we obtain that $$\ep_0^{\frac{2}{n-2}} +o(1)\leq \nu_\la^{-1}D_\la(\yl).$$ Hence, we have $\lim_{\la\to \ls} \nu_{\la}^{-1}(\yl-\xl)\neq 0$, thus $0\notin S$.  By \eqref{eq:wlabound}, for any compact subset $K\subset \rr^n\backslash S$ that contains $0$, there exists $C_K>0$ such that $$\big|w_\la(x)\big|\leq C_K \bb{ in } K.$$
 In particular, by standard elliptic theory, \eqref{critwla} and \eqref{haconvhi}  we get
\begin{equation}\label{limwlaw0}
w_{\la}\to w_0\in C^1_{\loc}(\rr^n\backslash S), 
\end{equation}
where $w_0$ verifies $-\Delta w_0= |w_0|^{2^*-2}w_0 \bb{ in } \rr^n\backslash S$, and $\big|w_0(0)\big|=1$. Independently, it follows from \eqref{sumvla} and \eqref{limnuBla} that $w_\la\to 0$ in $L^{2^*}(K)$ as $\la \to \ls$. Hence, by \eqref{limwlaw0} we find that 
\begin{equation*}
\int_K |w_0|^{2^*}\, dx =0.
\end{equation*}
Thus $w_0\equiv 0$ in $K$, which contradicts  $|w_0(0)|=1$. This ends the proof of Proposition \ref{limiteinLinfini}.
\end{proof}
For $\ro>0$ small enough, we define
\begin{equation}\label{etala}
	\eta_{\la}(\ro):=\sup_{\Om\backslash B_{\ro}(\xl)} \big| v_{\la}(x)\big|,
\end{equation} 
where $\xl$ is given by \eqref{defma}. Thanks to \eqref{eq:limiteinLinfini}, we obtain that
\begin{equation}\label{limetaladelta}
	\lim_{\la\to\ls} \sup \eta_{\la}(\ro)\leq \vv \vo\vv_{\infty}.
\end{equation} 
The next results establishes a first pointwise control on $\va$:

\begin{prop}\label{firstestimation}
For any $\nu \in (0,\frac{1}{2})$ there exists $R_{\nu}>0$, $\ro_{\nu}>0$, and $C_{\nu}>0$ such that for all $\alpha\in \nn$
\begin{equation}\label{eq:firstestimation}
\big|v_{\la}(x)\big|\leq C_{\nu}\left( \frac{\ma^{\frac{n-2}{2}-\nu(n-2)}}{|x-\xl|^{(n-2)(1-\nu)}}+\frac{\eta_{\la}(\ro_{\nu})}{|x-\xl|^{(n-2)\nu}}\right) 
\end{equation} 
for all $ x\in \Om \backslash B_{R_{\nu}\ma}(\xl)$.
\end{prop}

\begin{proof}
We divide our proof into two cases, depending on the position of $x_\infty$ with respect to the boundary of $\Om$.

\smallskip \noindent {\bf Case 1:} If $x_{\infty}\in\partial\Om$. Let $U \subset \R^n$ be a smooth bounded open set such that $\overline{\Om}\subset\subset U$. For all $\alpha \ge 1$, we extend $h_{\alpha}$ and $h_{\infty}$ as functions on $U$ in such a way that
\begin{equation}\label{haconvhi0}
	h_{\alpha}\to h_{\infty} \bb{ in }C^0(\overline{U})
\end{equation} 
and $- \Delta + h_\infty$ is still coercive in $H^1_0(U)$. Let $\tg:\overline{U}\times \overline{U}\backslash \{ (x,x): x\in \overline{U}\}\to \rr$ be the Green's function of the operator $-\Delta +h_{\infty}$ with Dirichlet boundary conditions in $U$. It exists by coercivity of $-\Delta + h_\infty$ and satisfies, for all $x \in U$,
\begin{equation}\label{deftG0}
-\Delta \tg(x, \cdot) +h_{\infty} \tg(x,\cdot)= \delta_x \quad \text{ in } U \backslash \{x\}. 
\end{equation}
We now define $\tg_{\la}(x):=\tg(\xl,x)$ for all $x\in \overline{U}\backslash \{\xl\}$ and $\alpha\in \nn$. It follows from \cite{RobDirichlet} that there exists $C_1>0$ such that 
\begin{equation}\label{Grenest1}
0<\tg_{\la} (x) \leq C_1 |x-\xl|^{2-n} \bb{ for all } x\in \overline{U} \backslash \{\xl\}
\end{equation}
and that there exist $\ro>0$ and $C_2>0$ such that 
\begin{equation}\label{Grenest2}
\tg_{\la}(x)\geq C_2|x-\xl|^{2-n} \bb{ and } \frac{|\nabla \tg_{\la}(x)|}{|\tg_{\la}(x)|}\geq C_2 |x-\xl|^{-1}
\end{equation}
for all $x\in B_{\ro}(\xl)\backslash \{\xl\}\subset\subset U$. We define
\begin{equation}\label{Llaope}
L_{\la}:=-\Delta +h_{\alpha} -|v_{\la}|^{2^*-2}
\end{equation}
and for a fixed $\nu \in (0,1)$ we let, for $\alpha \in \nn$ and $x\in \overline{U} \backslash \{\xl\}$,
\begin{equation}\label{defpsi}
	\psi_{\nu,\la}(x):=\ma^{\frac{n-2}{2}-\nu(n-2)}\tg_{\la}(x)^{1-\nu}+\eta_{\la}(\ro)\tg_{\la}(x)^{\nu}.
\end{equation}
Straightforward computations using \eqref{haconvhi0} and \eqref{deftG0} show that 
\begin{equation*}
\frac{L_{\la}\psi_{\nu,\la}}{\psi_{\nu,\la}} \ge - 2 \Vert h_\infty\Vert_{\infty} + o(1) + \nu(1-\nu) \left|\frac{\nabla \tg_{\la}}{\tg_{\la}}\right|^2 -|v_{\la}|^{2^*-2}.
\end{equation*}
By using  \eqref{Grenest2} we get that 
\begin{equation}\label{estG1}
\frac{L_{\la}\psi_{\nu,\la}}{\psi_{\nu,\la}}\geq - 2 \Vert h_\infty\Vert_{\infty} + o(1) +\nu(1-\nu)\frac{C_2^2}{ |x-\xl|^{2}}-|v_{\la}|^{2^*-2}
\end{equation}
for all $x\in B_{\ro}(\xl)\backslash \{\xl\}\subset\subset U$, where $C_2$ is the constant appearing in \eqref{Grenest2}. Proposition \ref{limiteinLinfini} now shows that there exists $R_0>0$ such that for any $R>R_0$ and $x\in \Om\backslash B_{R\ma}(\xl)$ we have
\begin{equation}\label{x-xlvla}
	|x-\xl|^{2}\big|v_{\la}(x)\mp v_{\infty}(x)\big|^{2^*-2}\leq  \frac{\nu(1-\nu)C_2^2}{2^{2^*+1}},
\end{equation}
 for $\la$ sufficiently large. Hence, by \eqref{x-xlvla} we get  
\begin{eqnarray}\label{estx-xlvla}
|x-\xl|^2\big|v_{\la}(x)\big|^{2^*-2} \le \frac{\nu(1-\nu)C_2^2}{4}+2^{2^*-1}\ro^2\vv v_{\infty}\vv_{\infty}^{2^*-2}
\end{eqnarray}
for all $
x\in \left( B_{\ro}(\xl)\backslash B_{R\ma}(\xl)\right)\cap \Om$. Choose $\ro_0>0$ small enough such that for any $\ro \in (0,\ro_0)$ we have
\begin{equation}\label{est2vla}
2^{2^*-1} \ro^2\vv v_{\infty}\vv_{\infty}^{2^*-2} + 2 \rho^2 \Vert h_\infty\Vert_{\infty} \leq \frac{\nu(1-\nu)C_2^2}{4}.
\end{equation}
Combining \eqref{estx-xlvla} and \eqref{est2vla} in \eqref{estG1} we finally obtain that, for all $x\in \left( B_{\ro}(\xl)\backslash B_{R\ma}(\xl)\right)\cap \Om$,
\begin{equation} \label{eq:Lapsi}
 L_{\la}\psi_{\nu,\la} \ge \frac{1}{|x-\xl|^{2}}\left( o(\rho^2) + \frac{\nu(1-\nu)C_2^2}{2} \right)\psi_{\nu,\la} >0
\end{equation}
holds. Independently, it follows from \eqref{limtvaB0}, \eqref{etala} and \eqref{Grenest2} that there exists $C = C(R,\rho, \nu) >0$ such that 
\begin{equation}\label{estvla3}
	\big| v_{\la}(x)\big| \leq  C \psi_{\nu,\la}(x) \bb{ for all } x\in \partial \Big( \Big( B_{\ro}(\xl) \backslash B_{R\ma}(\xl)\Big) \cap \Om\Big).
\end{equation}
By \eqref{critvlambda} $\va$ satisfies $L_\alpha \va = 0$. Using \eqref{eq:Lapsi} and \eqref{estvla3} we thus have 
\begin{equation}\label{eq:critlambda}
\left\{\begin{array}{ll}
L_{\la} (C\psi_{\nu,\la})\geq 0=L_{\la} v_{\la} &\hbox{ in }\Big( B_{\ro}(\xl) \backslash B_{R\ma}(\xl)\Big) \cap \Om\\
C\psi_{\nu,\la}\geq  v_{\la}&\hbox{ on }\partial \Big( \Big( B_{\ro}(\xl) \backslash B_{R\ma}(\xl)\Big) \cap \Om\Big)\\
L_{\la}(C\psi_{\nu,\la})\geq 0=-L_{\la} v_{\la} &\hbox{ in }  \Big( B_{\ro}(\xl) \backslash B_{R\ma}(\xl)\Big) \cap \Om\\
C\psi_{\nu,\la}\geq  -v_{\la}&\hbox{ on }\partial \Big( \Big( B_{\ro}(\xl) \backslash B_{R\ma}(\xl)\Big) \cap \Om\Big).
\end{array}\right.
\end{equation}
Since 	$\psi_{\nu,\la}>0$ and $L_{\la}	\psi_{\nu,\la}>0$ the operator $L_{\la}$ satisfies the comparison principle on $\left(  B_{\ro}(\xl) \backslash B_{R\ma}(\xl)\right)  \cap \Om$ (see e.g. \cite{BNV}), and therefore 
\begin{equation*}
\big| v_{\la}(x)\big| \leq  C \psi_{\nu,\la}(x) \bb{ for all } x\in  \left(  B_{\ro}(\xl) \backslash B_{R\ma}(\xl)\right)  \cap \Om.
\end{equation*}
Using again \eqref{Grenest1} implies \eqref{eq:firstestimation} in this case.

\medskip

\smallskip\noindent{\bf Case 2:} If now $x_{\infty}\in\Om$. Let $G$ be the Green's function in $\Om$ of the operator $-\Delta +h_{\infty}$ with Dirichlet boundary conditions. For $x \in \Om \backslash \{\xl\}$ define $\tilde{G}_{\la}:=G(\xl,\cdot)$, which satisfies  
\begin{equation*} 
	-\Delta \tilde{G}_{\la} +h_{\infty}\tilde{G}_{\la}= 0 \bb{ in } \Om\backslash \{\xl\}.
\end{equation*}
Since $x_\infty \in \Om$, it follows from \cite{RobDirichlet} that there exists $C_3>0$ such that 
\begin{equation*} 
	0<\tilde{G}_{\la} (x) \leq C_3 |x-\xl|^{2-n} \bb{ for all } x\in \overline{\Om} \backslash \{\xl\}
\end{equation*}
and there exist $C_4>0$ and $\rho>0$ such that 
\begin{equation*} 
	\tilde{G}_{\la}(x)\geq C_4|x-\xl|^{2-n} \bb{ and } \frac{|\nabla \tilde{G}_{\la}(x)|}{|\tilde{G}_{\la}(x)|}\geq C_4 |x-\xl|^{-1},
\end{equation*}
for all $x\in B_{\rho}(\xl)\backslash \{\xl\}\subset\subset \Om$. Define, for a fixed $\nu \in (0,1)$, for $\alpha \in \nn$ and $x\in \overline{\Om} \backslash \{\xl\}$,
\begin{equation*}
	\psi_{\nu,\la}(x):=\ma^{\frac{n-2}{2}-\nu(n-2)}\tilde{G}_{\la}(x)^{1-\nu}+\eta_{\la}(\ro)\tilde{G}_{\la}(x)^{\nu}
\end{equation*}
and let again $L_{\la}=-\Delta +h_{\alpha} -|v_{\la}|^{2^*-2}$. Mimicking the arguments in Case $1$ we here again have $\psi_{\nu,\la} >0$ and $L_\alpha \psi_{\nu,\la} >0$ in $B_{\ro}(\xl) \backslash B_{R\ma}(\xl)$, and the proof of \eqref{eq:firstestimation} follows in a similar way. 
\end{proof}

The next results establishes a pointwise control from above on $\va$:

\begin{prop}\label{est3vla}
There exists $C>0$ such that 
\begin{equation}\label{est5vla}
\big|v_{\la}(x)\big|\leq C \left(  \ma^{\frac{n-2}{2}}D_{\la}(x)^{2-n}+\vv \vo \vv_{\infty}\right)
\end{equation}
for all $x\in \Om$.
\end{prop}

\begin{proof}
Recall that $D_{\la}(x)=\ma +|x-\xl|$ for $x \in \Om$. We first prove that there exists $\ro>0$ and $C>0$ such that 
\begin{equation}\label{estvlaDlaetala}
|v_{\la}(x)|\leq C\left( \ma^{\frac{n-2}{2}}D_{\la}(x)^{2-n}+\eta_{\la}(\rho)\right),
\end{equation}
where  $\eta_{\la}(\ro)$ is defined in \eqref{etala}. We fix $0<\nu <\frac{1}{n+2}$ and we let $R_{\nu}>0$ and $\ro_{\nu}>0$ be given by Proposition  \ref{firstestimation}. We let $\rho = \rho_\nu$. Proving \eqref{estvlaDlaetala} amounts to proving that for any sequence $\yl\in \Om$, we have
\begin{equation}\label{fracvlaO1}
\frac{\big|v_{\la}(\yl)\big|}{\ma^{\frac{n-2}{2}}D_{\la}(\yl)^{2-n}+\eta_{\la}(\ro)}=O(1) \bb{ as } \la \to \ls.
\end{equation}
We let in this proof $r_{\la}:=|\yl-\xl|$. First, if  $r_{\la}\geq \ro$, it  is clear that \eqref{fracvlaO1} is satisfied by definition of $\eta_{\la}(\ro)$. If now $r_\la=O(\ma)$ we also have $D_\la(\yl)=O(\ma)$ and \eqref{contprojbulle2} and \eqref{eq:limiteinLinfini} yield
\begin{eqnarray*}
	D_{\la}(\yl)^{n-2}\ma^{-\frac{n-2}{2}}\big|v_\la(\yl)\big|=O(1),
\end{eqnarray*}
which proves \eqref{fracvlaO1}. We thus assume from now on that 
\begin{equation}\label{case2}
r_{\la}\leq \ro \quad \bb{ and } \quad \lim_{\la\to\ls}\frac{r_{\la}}{\ma}=+\infty.
\end{equation}
 Green's representation formula and \eqref{estGlai} yield the existence of $C>0$ such that
\begin{equation} \label{est4vla}
\big|v_{\la}(\yl)\big| 
\leq C \int_{\Om} |\yl-x|^{2-n} \big|v_{\la}(x)\big|^{2^*-1}\, dx,
\end{equation}
for all $\alpha \ge 1$. We write that
\begin{eqnarray}\label{est4vla1}
	\int_{\Om} |\yl-x|^{2-n} \big|v_{\la}(x)\big|^{2^*-1}\, dx
	&\leq& \int_{\Om \cap \{|x-\xl|\leq R_{\nu} \ma\}}|\yl-x|^{2-n} |v_{\la}(x)|^{2^*-1}\, dx\nonumber\\
	&&+\int_{\Om \cap \{|x-\xl|\geq R_{\nu} \ma\}}|\yl-x|^{2-n} |v_{\la}(x)|^{2^*-1}\, dx.
\end{eqnarray} 
Fix $C_0>R_{\nu}$. For $\la$ sufficiently large we have using \eqref{case2} that 
\begin{equation*}
r_{\la}\geq C_0 \ma \geq \frac{C_0}{\Ru}|x-\xl| \bb{ for all } x\in \Om \cap \left\lbrace |x-\xl|\leq \Ru \ma\right\rbrace,
\end{equation*}
so that $|\yl-x|\geq\big( 1-\Ru C_0^{-1}\big) r_{\la}$ for all such $x$. Therefore,  using H\"older's inequality and \eqref{limvlaL2star} yields
\begin{equation}\label{firstterm}
\begin{aligned}
&	\int_{\Om \cap \left\lbrace |x-\xl|\leq \Ru \ma\right\rbrace }|\yl-x|^{2-n} \big|v_{\la}(x)\big|^{2^*-1}\, dx\\
& =  O\left( \frac{\ma^{\frac{n-2}{2}}}{|\yl-\xl|^{n-2}}\right).
\end{aligned}
\end{equation}
Now, we deal with the second term of \eqref{est4vla1}. From \eqref{eq:firstestimation}, we get 
\begin{equation}\label{sect}
\begin{aligned}
&	\int_{\Om \cap \{|x-\xl|\geq \Ru \ma\}}|\yl-x|^{2-n} \big|v_{\la}(x)\big|^{2^*-1}\, dx\nonumber\\
&= O \left( \ma^{\frac{n+2}{2}(1-2\nu)}\int_{\Om \cap \{|x-\xl|\geq \Ru \ma\}}   \frac{|\yl-x|^{2-n}}{|x-\xl|^{(n+2)(1-\nu)}}\, dx \right) \\
&+O \left( \eta_{\la}(\ro_{\nu})^{2^*-1}\int_{\Om \cap \{|x-\xl|\geq \Ru \ma\}}\frac{|\yl-x|^{2-n}}{|x-\xl|^{(n+2)\nu}} \, dx \right).\nonumber
\end{aligned} 
\end{equation} 
Since $2-(n+2)\nu>0$, using Giraud's lemma (see \cite[Lemma 7.5]{HebeyZLAM}) yields
\begin{equation}\label{Giraud1}
\int_{\Om}|\yl-x|^{2-n}|x-\xl|^{-(n+2)\nu} \, dx
=O(1).
\end{equation}
Independently, letting $\tilde{y}_{\alpha} = \frac{y_\alpha - \xl}{\ma}$ we have 
\begin{equation} \label{secondtermbis}
\begin{aligned}
&\int_{\Om \cap \{|x-\xl|\geq \Ru\ma\}}   \frac{1}{|\yl-x|^{n-2}}\frac{1}{|x-\xl|^{(n+2)(1-\nu)}}\, dx\\
&  \le \ma^{2 - (n+2)(1-\nu)} \int_{\R^n \backslash B(0,R_\nu)}  \frac{1}{|\tilde{y}_\alpha - x|^{n-2}} \frac{1}{|x|^{(n+2)(1-\nu)}}\, dx \\
& = O \left( \frac{\ma^{2 - (n+2)(1-\nu)}}{(1+ |\tilde{y}_{\alpha}|)^{n-2}}\right) =  O \left( \frac{\ma^{n - (n+2)(1-\nu)}}{|x_\alpha - y_\alpha|^{n-2}}\right),
\end{aligned}
\end{equation}
where the third line again follows from Giraud's lemma in $\R^n$ since $(n+2)(1-\nu) > n$. Combining \eqref{Giraud1} and \eqref{secondtermbis} finally shows that 
$$ 	\int_{\Om \cap \{|x-\xl|\geq \Ru \ma\}}|\yl-x|^{2-n} \big|v_{\la}(x)\big|^{2^*-1}\, dx = O \left( \frac{\ma^{\frac{n-2}{2}}}{|x_\alpha - y_\alpha|^{n-2}}\right) + O(\eta_\alpha(\rho)),$$
which together with  \eqref{firstterm} concludes the proof of \eqref{estvlaDlaetala}.

We now conclude the proof of \eqref{est5vla}. First, if $v_\infty >0$, \eqref{est5vla} simply follows from \eqref{limetaladelta} and \eqref{estvlaDlaetala}. We may thus assume that $v_\infty \equiv 0$. We now prove that for $\alpha$ large enough 
\begin{equation} \label{controlerhoa}
\eta_\alpha(\rho) = O\big( \ma^{\frac{n-2}{2}} \big)
\end{equation}
holds. Together with  \eqref{estvlaDlaetala} this will conclude the proof of  \eqref{est5vla} in this case. We prove \eqref{controlerhoa} by contradiction: we assume that 
\begin{equation} \label{controlerhoa2}
\frac{\eta_\alpha(\rho)}{\ma^{\frac{n-2}{2}}} \to + \infty 
\end{equation}
as $\alpha \to + \infty$, and we let $V_\alpha = \frac{v_\alpha}{\eta_\alpha(\rho)}$. For any $\alpha$ we let $z_\alpha \in \Om\backslash B_{\ro}(\xl)$ be such that $|v_\alpha(z_\alpha)| = \eta_\alpha(\rho)$. By the definition of $D_\alpha(x)$ and by  \eqref{estvlaDlaetala} we see that for any $\delta >0$ fixed we have $|V_\alpha(z_\alpha)|=1$ and 
\begin{equation} \label{convva}
|V_\alpha(x)| \le C + o(1) \quad \text{ for } x \in \Om \backslash B_{\delta}(\xl). 
\end{equation}
Now, the function $V_\alpha$ satisfies 
\begin{equation*}
-\Delta V_\alpha +h_\alpha V_\alpha  =\eta_\alpha(\rho)^{2^*-2}|V_\alpha|^{2^*-2} V_\alpha
\end{equation*} 
in $\Om$. Since $\eta_\alpha(\rho) \to 0$ by  \eqref{limetaladelta}, \eqref{convva} and standard elliptic theory show that $V_\alpha \to V_\infty$ in $C^2_{loc}(\overline{\Om} \backslash \{x_\infty\}$ as $\alpha \to + \infty$, where $V_\infty$ satisfies $|V_\infty(x)| \le C$ for any $x \neq x_\infty$ and 
\begin{equation*}
-\Delta V_\infty +h_\infty V_\infty  =0 \quad \text{ in }  \Om \backslash \{x_\infty\}. 
\end{equation*} 
In particular, the singularity of $V_\infty$ at $x_\infty$ is removable and $V_\infty$ satisfies weakly $-\Delta V_\infty +h_\infty V_\infty  =0$ in $\Om$. Since $- \Delta + h_\infty$ is coercive by assumption, this shows that $V_\infty \equiv 0$. Independently, if we let $z_\infty = \lim_{\alpha \to + \infty} z_\alpha$, the $C^2_{loc}$ convergence shows that $|V_\infty(z_\infty)| = 1$, hence $V_\infty \not \equiv 0$. This is a contradiction, which concludes the proof of \eqref{controlerhoa}.
\end{proof}

The next result is will be frequently used in the proof of Theorem \ref{maintheo1}:
\begin{prop}\label{lemme:estu0}
	Let $U \subset \Om$ be an open set. There exists a constant $C(U)$ such that $\lim_{|U| \to 0} C(U) = 0$ and such that, for all $y\in \Om$ and for all $\alpha \ge 1$,
	\begin{equation}\label{eq:estu0}
		\int_{U} G_{\alpha}(y,x)\, dx\leq C(U)\,  d(y,\partial\Om).
	\end{equation}
\end{prop}
\begin{proof}
We let $C(U) = \sup_{y \in \Om} \int_{U} |x-y|^{1-n} \, dx$. Since $\Om$ is bounded and $y \mapsto |y|^{1-n} \in L^{1}_{loc}(\R^n)$ we have $C(U) \to 0$ as $|U| \to 0$ by absolute continuity of the integral. Using \eqref{estGlai} yields
\begin{equation}\label{sumI1+I2}
	\int_{U} G_{\alpha}(y,x)\, dx= 
	O\left(  I_{1}(y)+I_{2}(y)\right) 
\end{equation}
where we have let, for $i=1,2$, 
\begin{equation*}
	I_i(y):= \int_{U_i}\frac{1}{|y-x|^{n-2}} \min\Big\{1,\frac{d(y,\partial\Om)d(x,\partial\Om)}{|y-x|^2}\Big\}\, dx,
\end{equation*}
and
\begin{eqnarray*}
	U_1:=U \cap \Big\{|y-x|<\frac{d(y,\partial\Om)}{2}\Big\} \bb{ and }U_2:=U \cap \Big\{|y-x|>\frac{d(y,\partial\Om)}{2}\Big\}.
\end{eqnarray*}
When $x \in U_1$ we have $|y-x|<\frac{d(y,\partial\Om)}{2}$ so that
\begin{equation*}
\begin{aligned}
I_1(y)\le\int_{U_1}\frac{1}{|y-x|^{n-2}}\, & \le \frac{d(y,\partial\Om)}{2}\int_{U}\frac{1}{|y-x|^{n-1}}\, dx \\
&  \le \frac{C(U)}{2} d(y, \partial \Om). 
\end{aligned}
\end{equation*}
When $x\in U_2$ we get that $d(x,\partial \Om) \leq 3|y-x|$.  We then get that 
\begin{equation*}
\begin{aligned}
I_2(y) \le d(y,\partial\Om)\int_{U_2} \frac{d(x,\partial\Om)}{|y-x|^{n}}
					& \le 3 d(y,\partial\Om) \int_{U}\frac{1}{|y-x|^{n-1}}\, dx \\
					& \le 3 C(U) d(y, \partial \Om). 
\end{aligned}
\end{equation*}
Combining these estimates proves Proposition \ref{lemme:estu0}.\end{proof}

The next result improves the upper estimate in Proposition \ref{est3vla}: 

\begin{prop}
There exists $C>0$ such that 
\begin{equation}\label{controlva}
		\big| v_{\la}(x)\big| \leq \, C\left(  B_{\alpha}(x)+\vo(x)\right) \bb{ for all } \alpha \bb{ and all } x\in \Om.
	\end{equation}
\end{prop}

\begin{proof}
First, if $v_\infty \equiv 0$, \eqref{controlva} simply follows from \eqref{est5vla}. We may thus assume in the following that $v_\infty >0$ in $\Om$.
Proving  \eqref{controlva} in Theorem \ref{maintheo1} is equivalent to proving that for any sequence $(\yl)_{\alpha\in\nn}\in \Om$, we have
			\begin{equation}\label{fracvlaO1+}
				\frac{\big|v_{\la}(\yl)\big|}{B_{\alpha}(\yl)+\vo(\yl)}=O(1) \bb{ as } \la \to \ls.
			\end{equation}
Assume first that $|\yl-\xl|=O(\ma)$. It follows from \eqref{contprojbulle2} and Proposition \ref{limiteinLinfini} that 
\begin{equation*}
\left| v_\la(\yl)\right| =O\left(\vo(\yl)+ B_\la(\yl)\right) +o\left( D_\la(\yl)^{-\frac{n-2}{2}}\right) =  O\left(\vo(\yl)+ B_\la(\yl)\right),
\end{equation*}
which proves \eqref{controlva} in this case. We thus assume from now on that 
\begin{equation}\label{limDlama}
\lim_{\la\to\ls} \frac{|\yl-\xl|}{\ma}=+\infty.
\end{equation}
Using Proposition \ref{limiteinLinfini} and standard elliptic theory, we have that
\begin{equation}\label{vatou0}
v_\la\to \mp\vo \bb{ in }C_{loc}^2(\overline{\Om}\backslash \{\xin\})  \bb{ as } \la\to \ls.
\end{equation}
Therefore, there exists $\ro_\la>0$, $\ro_\la \to 0$ as $\la\to \ls$, such that, up to a subsequence 
\begin{equation}\label{limvlau0C2}
\vv v_{\la}\pm\vo\vv_{C^2(\{|x-\xl|> \ro_\la\}\cap\Om)}=o(1).
\end{equation}
 Using again Green's representation formula and \eqref{estGlai} we have 
\begin{equation} \label{reprefinaleTh1}
\begin{aligned}
\big| v_{\la}(\yl)\big|  
& = O\Bigg( \int_{\{|x-\xl|\leq \ro_\la\}\cap\Om} G_{\la}(\yl,x)|v_{\la}(x)|^{2^*-1}\, dx \\
&+ \int_{\{|x-\xl|> \ro_\la\}\cap\Om} G_{\la}(\yl,x)|v_{\la}(x)|^{2^*-1}\, dx \Bigg) .
\end{aligned}
\end{equation}
Thanks to \eqref{estu0}, \eqref{eq:estu0} and \eqref{limvlau0C2}, we get that 
\begin{equation} \label{vla}
  \int_{\{|x-\xl|> \ro_\la\}\cap\Om} G_{\la}(\yl,x)|v_{\la}(x)|^{2^*-1}\, dx  =O\left( \vo(\yl)\right). 
  \end{equation}
 We fix $R>0$, and we now write the following
\begin{equation} \label{decomp}
\begin{aligned}
&	\int_{\Om\cap \{|x-\xl|\leq \ro_{\la}\}}G_\la(\yl,x)\big|v_{\la}(x)\big|^{2^*-1}\, dx \\
&=O\Bigg( \int_{\Om\cap  \{|x-\xl|\leq R\ma\}}|\yl-x|^{2-n}\big|v_{\la}(x)\big|^{2^*-1}\, dx \\
	& + \int_{\Om\cap \{R \ma \le |x-\xl|\leq \ro_{\la}\}}G_\la(\yl,x)\big|v_{\la}(x)\big|^{2^*-1}\, dx \Bigg).
\end{aligned}
\end{equation}
As in the proof of \eqref{firstterm}, thanks to \eqref{limvlaL2star} and to H\"older's inequality, we obtain 
\begin{equation}\label{fdecomp}
\int_{\Om\cap \{|x-\xl|\leq R\ma\}}|\yl-x|^{2-n} \big|v_{\la}(x)\big|^{2^*-1}\, dx	=O\left( \frac{\ma^{\frac{n-2}{2}}}{|\yl-\xl|^{n-2}}\right).
\end{equation}
\noindent By \eqref{est5vla}, there exists $C>0$ such that 
\begin{equation*} 
	\big|v_{\la}(x)\big|^{2^*-1}\leq C\Big( \ma^{\frac{n+2}{2}}D_{\la}(x)^{-2-n}+\vv \vo \vv_{\infty}^{2^*-1} \Big),
\end{equation*}
where $D_{\la}(x):=\ma +|x-\xl|$  for all $x\in \Om$. Therefore, using again \eqref{estu0}, we have
\begin{eqnarray}\label{sdecomp}
	&&\int_{\Om\cap \{R \ma \le |x-\xl|\leq \ro_{\la}\}}G_\la(\yl,x)\big|v_{\la}(x)\big|^{2^*-1}\, dx\nonumber\\
	&=&O\Big(\ma^{\frac{n+2}{2}}\int_{\Om \cap \{|x-\xl|\geq R \ma\}  }|\yl-x|^{2-n}|x-\xl|^{-2-n}\, dx \Big)\nonumber\\
	&&+O\left(\int_{\Om\cap \{R \ma \le |x-\xl|\leq \ro_{\la}\}} G_\la(\yl,x)\, dx \right)\nonumber \\
	&=&O\left( \frac{\ma^{\frac{n-2}{2}}}{|x_\alpha - y_\alpha|^{n-2}}\right)  + O(\vo(\yl)).
\end{eqnarray}
Combining \eqref{fdecomp} and \eqref{sdecomp} in \eqref{decomp} finally shows that 
\begin{equation*} 
	\int_{\Om\cap \{|x-\xl|\leq \ro_{\la}\}}G_\la(\yl,x)\big|v_{\la}(x)\big|^{2^*-1}\, dx=O\Big(\ma^{\frac{n-2}{2}}|x_\alpha - y_\alpha|^{2-n} \Big)+ O(\vo(\yl))
\end{equation*}
as $\alpha \to + \infty$. Together with \eqref{vla} and \eqref{fdecomp} this proves \eqref{fracvlaO1+} and concludes the proof of \eqref{controlva}. 
 \end{proof}

We are now in position to conclude the proof of Theorem \ref{maintheo1}:

\begin{proof}[Proof of Theorem \ref{maintheo1}]
Proving Theorem \ref{maintheo1} is equivalent to proving that for any sequence $(\yl)_{\alpha\in\nn}\in \Om$, we have
\begin{equation}\label{DLfinal}
v_\alpha(y_\alpha) = \Pi B_\alpha(v_\alpha) \pm v_\infty(y_\alpha) + o \big( B_\alpha(y_\alpha) \big) + o \big( v_\infty(y_\alpha) \big)
\end{equation}
as $\alpha \to + \infty$. Throughout this proof it will be intended that all the terms involving $v_\infty$ disappear if $v_\infty \equiv 0$. If $|x_\alpha - y_\alpha| = O(\ma)$ or if $|x_\alpha - y_\alpha| \not \to 0$, \eqref{DLfinal} follows from Proposition \ref{limiteinLinfini}. We may thus assume in the following that
\begin{equation}\label{case2Th1}
|x_\alpha-y_\alpha| \to 0 \quad \bb{ and } \quad \frac{|x_\alpha-y_\alpha|}{\ma} \to +\infty
\end{equation}
as $\alpha \to + \infty$. We write three representation formulae for $v_\alpha, \Pi B_\alpha$ and $v_\infty$, using respectively \eqref{critvlambda}, \eqref{critu0} and \eqref{projbulle} and we substract them to get: 
\begin{equation} \label{reprefinaleTh1}
\begin{aligned}
&v_\alpha(y_\alpha) - \Pi B_\alpha(y_\alpha) \mp v_\infty(y_\alpha) \\
& = \int_{\Om} G_\alpha(y_\alpha, \cdot)\Big( |v_\alpha|^{2^*-2} v_\alpha - B_\alpha^{2^*-1} \mp v_\infty^{2^*-1} \Big)\, dx \\
& \pm \int_{\Om} \Big( G_\alpha(y_\alpha, \cdot) - G_\infty(y_\alpha, \cdot) \Big)  v_\infty^{2^*-1}  \, dx, 
\end{aligned} 
\end{equation}
where we have denoted by $G_\infty$ the Green's function for $-\Delta + \hi$. 

\medskip

\textbf{We assume first that $v_\infty \equiv 0$.} In this case the second integral in \eqref{reprefinaleTh1} vanishes and we only have to estimate the first one. Let $R >1$ be fixed. Using \eqref{estGlai}, \eqref{est5vla} and letting $\check{y}_\alpha = \frac{y_\alpha - x_\alpha}{\ma}$ a simple change of variables and direct computations give 
\begin{equation} \label{theorieC01} 
\begin{aligned} 
&\Big|\int_{\Om \backslash B_{R \ma}(x_\alpha)} G_\alpha(y_\alpha, \cdot)\Big( |v_\alpha|^{2^*-2} v_\alpha - B_\alpha^{2^*-1}  \Big)\, dx \Big|\\
&  \le C \ma^{-\frac{n-2}{2}} \int_{\R^n \backslash B_R(0)} \frac{1}{|\check{y}_\alpha-x|^{n-2}} B_0^{2^*-1} \, dx \\
& = O\big( \ve_R B_\alpha(y_\alpha) \big)
\end{aligned} 
\end{equation} 
as $\alpha \to + \infty$, where $\ve_R$ denotes a positive number satisfying $\lim_{R \to + \infty} \ve_R = 0$. Independently,  \eqref{contprojbulle2} and \eqref{limtvaB0} show that 
$$ \left \Vert \frac{v_\alpha - B_\alpha}{B_\alpha} \right \Vert_{L^\infty(B_{R\ma}(x_\alpha))} \to 0 $$
as $\alpha \to + \infty$. As a consequence, and with \eqref{estGlai},
\begin{equation} \label{theorieC02}
 \begin{aligned}
&\Big| \int_{B_{R \ma}(x_\alpha)} G_\alpha(y_\alpha, \cdot)\Big( |v_\alpha|^{2^*-2} v_\alpha - B_\alpha^{2^*-1}  \Big)\, dx \Big| \\
&  = o \Big( \int_{B_{R \ma}(x_\alpha)}  |y_\alpha-y|^{2-n} B_\alpha^{2^*-1}\, dx \Big)\\
& = o \big( B_\alpha(y_\alpha) \big). 
\end{aligned} \end{equation}
Up to passing to a subsequence, combining \eqref{theorieC01} and \eqref{theorieC02} proves \eqref{DLfinal} in the $v_\infty \equiv 0$ case. 

\medskip

\textbf{We now assume that $v_\infty > 0$.} We first estimate the first integral in \eqref{reprefinaleTh1} by decomposing it in three domains: $B_{R\ma}(x_\alpha)$, $\big( \Om \cap B_{\frac{1}{R}}(x_\alpha)\big)\backslash B_{R\ma}(x_\alpha)$ and $\Om \backslash B_{\frac{1}{R}}(x_\alpha)$. We first have 
\begin{equation} \label{theorieC03}
 \begin{aligned}
& \int_{B_{R \ma}(x_\alpha)} G_\alpha(y_\alpha, \cdot)\Big( |v_\alpha|^{2^*-2} v_\alpha - B_\alpha^{2^*-1}  \mp v_\infty^{2^*-1} \Big)\, dx \\
&  =  \int_{B_{R \ma}(x_\alpha)} G_\alpha(y_\alpha, \cdot)\Big( |v_\alpha|^{2^*-2} v_\alpha - B_\alpha^{2^*-1}  \Big)\, dx \\
& + O \Big(  \int_{B_{R \ma}(x_\alpha)} G_\alpha(y_\alpha, \cdot) \, dx \Big)\\
& = o \big( B_\alpha(y_\alpha) \big) + o\big( v_\infty(y_\alpha) \big), 
\end{aligned} \end{equation}
where the last line follows from \eqref{theorieC02} and from \eqref{estu0} and \eqref{eq:estu0} with $U = B_{R \ma}(x_\alpha)$. Using \eqref{limvlau0C2} we now have 
\begin{equation} \label{theorieC04}
 \begin{aligned}
& \int_{\Om \backslash B_{\frac{1}{R}}(x_\alpha)} G_\alpha(y_\alpha, \cdot)\Big( |v_\alpha|^{2^*-2} v_\alpha - B_\alpha^{2^*-1} \mp v_\infty^{2^*-1}   \Big)\, dx \\
&  =  \int_{\Om \backslash B_{\frac{1}{R}}(x_\alpha)} G_\alpha(y_\alpha, \cdot)\Big( |v_\alpha|^{2^*-2} v_\alpha  \mp v_\infty^{2^*-1} \Big)\, dx + O \big( \ma^{\frac{n+2}{2}} \big) \\
& = o \Big( \int_{\Om } G_\alpha(y_\alpha, y) \, dy \Big) + o \big(B_\alpha(y_\alpha) \big)  \\
& = o \big(B_\alpha(y_\alpha) \big)+ o\big( v_\infty(y_\alpha) \big), 
\end{aligned} \end{equation}
where the last line again follows from \eqref{estu0} and \eqref{eq:estu0}. Finally, using \eqref{estGlai} and  \eqref{est5vla} we have 
\begin{equation} \label{theorieC05}
 \begin{aligned}
&\Big| \int_{ (\Om \cap B_{\frac{1}{R}}(x_\alpha))\backslash B_{R\ma}(x_\alpha)} G_\alpha(y_\alpha, \cdot)\Big( |v_\alpha|^{2^*-2} v_\alpha - B_\alpha^{2^*-1} \mp v_\infty^{2^*-1}   \Big)\, dx \Big| \\
&  = O \Big( \int_{\Om \backslash B_{R \ma}(x_\alpha)} |y_\alpha-y|^{2-n} B_\alpha^{2^*-1} \, dx \Big) + O \Big( \int_{\Om \cap B_{\frac{1}{R}}(x_\alpha)}G_\alpha(y_\alpha, y) \, dy \Big) \\
& = O \big(\ve_R B_\alpha(y_\alpha) \Big)+ O\big( \ve_R v_\infty(y_\alpha) \big),
\end{aligned} \end{equation}
where the last line follows from \eqref{theorieC01} and  \eqref{eq:estu0} with $U = \Om \cap B_{\frac{1}{R}}(x_\alpha)$. Combining \eqref{theorieC03}, \eqref{theorieC04} and \eqref{theorieC05} proves that 
\begin{equation} \label{theorieC06}
 \begin{aligned}
& \int_{\Om} G_\alpha(y_\alpha, \cdot)\Big( |v_\alpha|^{2^*-2} v_\alpha - B_\alpha^{2^*-1} \mp v_\infty^{2^*-1} \Big)\, dx  \\
& = o \big(B_\alpha(y_\alpha) \big)+ o\big( v_\infty(y_\alpha) \big) + O \big(\ve_R B_\alpha(y_\alpha) \Big)+ O\big( \ve_R v_\infty(y_\alpha) \big)
\end{aligned} \end{equation}
as $\alpha \to + \infty$, where $\lim_{R \to + \infty} \ve_R = 0$. We now estimate the second integral in \eqref{reprefinaleTh1}. For $y \in \Om$ and for all $\alpha$, we let 
$$\begin{aligned}
 F_{1,\alpha}(y) & = \int_{\Om} G_\alpha(y, \cdot)  v_\infty^{2^*-1}  \, dx \text{ and } \\
 F_{2}(y) & = \int_{\Om}  G_\infty(y, \cdot) v_\infty^{2^*-1}  \, dx \\
 \end{aligned} $$
By definition of $G_\alpha$ and $G_\infty$, these functions satisfy respectively $(-\Delta + h_\alpha)F_{1,\alpha} = v_\infty^{2^*-1}$ and $(-\Delta + h_\infty)F_{2} = v_\infty^{2^*-1}$, so that by \eqref{haconvhi} and standard elliptic theory $(F_{1,\alpha})_{\alpha \in \nn}$ is uniformly bounded in $L^\infty(\Om)$. We also have 
$$ \big( - \Delta + h_\infty \big)(F_{1,\alpha} - F_2) = (h_\infty - h_\alpha) F_{1,\alpha}. $$
A representation formula for $F_{1,\alpha} - F_2$ applied at $y_\alpha$ then shows that 
$$ \begin{aligned}
 \int_{\Om} \Big( G_\alpha(y_\alpha, \cdot) - G_\infty(y_\alpha, \cdot) \Big)  v_\infty^{2^*-1}  \, dx & = F_{1,\alpha}(y_\alpha) - F_2(y_\alpha) \\
& = \int_{\Om} G_\infty(y_\alpha, \cdot)  (h_\infty - h_\alpha) F_{1,\alpha} \, dx.
\end{aligned} $$
Using \eqref{haconvhi}, \eqref{estu0} and \eqref{eq:estu0} we thus obtain 
 \begin{equation} \label{theorieC07}
 \begin{aligned} 
  \Big| \int_{\Om} \Big( G_\alpha(y_\alpha, \cdot) - G_\infty(y_\alpha, \cdot) \Big)  v_\infty^{2^*-1}  \, dx \Big|   & = o \Big( \int_{\Om} G_\infty(y_\alpha, x)\, dx \Big) \\
 &  = o\big( v_\infty(y_\alpha) \big).
\end{aligned} \end{equation}
Plugging \eqref{theorieC06} and \eqref{theorieC07} in \eqref{reprefinaleTh1} finally proves that 
$$
\begin{aligned}
 \Big| v_\alpha(y_\alpha) - \Pi B_\alpha(y_\alpha) \mp v_\infty(y_\alpha) \Big| & = o \big(B_\alpha(y_\alpha) \big)+ o\big( v_\infty(y_\alpha) \big) \\
 & + O \big(\ve_R B_\alpha(y_\alpha) \Big)+ O\big( \ve_R v_\infty(y_\alpha) \big)
\end{aligned}$$
as $\alpha \to + \infty$, where $\lim_{R \to + \infty} \ve_R = 0$. Passing to a subsequence proves \eqref{DLfinal} and concludes the proof of Theorem \ref{maintheo1}.
\end{proof}
 
 	\section{Necessary conditions for blow-up and proof of Theorem \ref{maintheo2bis}}\label{secpohozaev}

Let $\Om$ be a smooth bounded domain of $\rr^n$, $n\geq 3$. Throughout this section we let $(\ha)_{\alpha\in\nn}$ be a sequence of functions that converges in $C^{1}(\overline{\Om})$ to $\hi$, where $-\Delta+\hi$ is coercive in $H^1_0(\Om)$ and where $I_{\hi}(\Om)<K_{n}^{-2}$, and we let $(\va)_{\alpha\in\nn}\in H_0^1(\Om)$ be a sequence of solutions of \eqref{critvlambda} that satisfies \eqref{limvlaL2star}, \eqref{limvlaLinfini} and \eqref{sumvla}. Equation \eqref{sumvla2} is thus also satisfied and 
we have 
 \begin{equation*} 
 	\va=\Pi B_{\alpha}\pm\vo+o(1)  \bb{ in } H_0^1(\Om) \text{ as } \alpha \to + \infty,  
\end{equation*}
where $\Pi B_\alpha$ is given by \eqref{projbulle} and where $(x_\alpha)_{\alpha \in \nn}$ and $(\ma)_{\alpha \in \nn}$ are sequences of points in $\Om$ and $(0, + \infty)$ satisfying \eqref{defma} and with $\lim_{\alpha \to + \infty} \ma = 0$. We let again $x_\infty = \lim_{\alpha \to + \infty} x_\alpha$ and we identify in this section necessary blow-up conditions that constrain the localisation of $x_\infty$. We recall for this the celebrated Pohozaev identity, that for our sequence $(\va)_{\alpha\in\nn}$ is as follows: for any family $U_{\alpha}$ of smooth domains such that $\xl\in U_{\alpha}\subset \Om$ for $\alpha\in\nn$ we have 
 \begin{eqnarray}\label{poa}
	&&\int_{U_{\alpha}}\left(\ha(x)+\frac12 \l\nabla \ha(x),x-\xl\r \right)  \va^2\, dx\nonumber\\
		&=&\int_{\partial U_{\alpha}} \l x-\xl, \nu \r\left(  \frac{|\nabla \va|^2}{2}+\ha \frac{\va^2}{2}
		-\frac{|\va|^{2^*}}{2^*}\right) \,  d\sigma(x)\\
		&&-\int_{\partial U_{\alpha}}\left( \l x-\xl, \nabla \va \r +\frac{n-2}{2}\va\right)\partial_{\nu}\va \,  d\sigma(x),\nonumber
	\end{eqnarray}
where $\nu$ is the outer unit normal to the boundary of $U_\alpha$ and $\l \cdot, \cdot \r$ is the Euclidean scalar product (see for instance \cite[Lemma 6.5]{HebeyZLAM}). We distinguish two cases according to whether $x_\infty$ is a boundary blow-up point or not.

\subsection{Interior blow-up case: $x_\infty \in \Om$} If $x_\infty$ is an interior point we prove the following result:

\begin{prop} \label{prop:blowup:interieur}
Let $\Om$ be a smooth bounded domain of $\rr^n$, $n\geq 3$. Let $(\ha)_{\alpha\in\nn}$ be a sequence of functions that converges in $C^{1}(\overline{\Om})$ to $\hi$, where $-\Delta+\hi$ is coercive in $H^1_0(\Om)$ and where $I_{\hi}(\Om)<K_{n}^{-2}$, and we let $(\va)_{\alpha\in\nn}\in H_0^1(\Om)$ be a sequence of solutions of \eqref{critvlambda} that satisfies \eqref{limvlaL2star}, \eqref{limvlaLinfini} and \eqref{sumvla}. Let $x_\infty = \lim_{\alpha \to + \infty} x_\alpha$ and assume that $x_\infty \in \Om$. Then
\begin{itemize}
\item If $n=3$: we have $v_\infty \equiv 0$ and $m_{h_\infty}(x_\infty) = 0$.
\item If $n=4,5$: we have $v_\infty \equiv 0$ and $h_\infty(x_\infty) = 0$.
\item If $n=6$, we have $h_\infty(x_\infty) = \pm 2 v_\infty(x_\infty)$.
\item If $n \ge 7$, we have $h_\infty(x_\infty) = 0$. 
	\end{itemize}
\end{prop}

\begin{proof}
First, since $x_\infty \in \Om$, we have $\bsa(\xl)\subset \Om$ for all $\alpha$ large enough. The Pohozaev Identity \eqref{poa} yields 
	\begin{eqnarray}\label{poa1}
		&&\int_{\bsa(\xl)}\left(\ha(x)+\frac12 \l\nabla \ha(x),x-\xl \r\right)  \va^2\, dx =\int_{\partial \bsa(\xl)}  F_{\alpha}(x) \,  d\sigma(x),
	\end{eqnarray}
	where we have let 
	\begin{equation}\label{balph}
	\begin{aligned}
	F_{\alpha}(x)& :=\l x-\xl, \nu\r\left(  \frac{|\nabla \va|^2}{2}+\ha \frac{\va^2}{2}-\frac{|\va|^{2^*}}{2^*}\right) \\
	& -\left( \l x-\xl, \nabla \va \r+\frac{n-2}{2}\va \right)\partial_{\nu}\va.
\end{aligned} 
	\end{equation} 
For any $x\in \frac{\Om-\xl}{\sqrt{\ma}}$ we let 
		\begin{equation*} 
			\hva(x)= \va (\xl+\sqrt{\ma} x). 
		\end{equation*}
Using \eqref{critvlambda} it is easily seen that $\hva$ satisfies 
		\begin{equation*} 
			\left\{\begin{array}{ll}
				-\Delta \hva +\ma\hha \hva=\ma\left|\hva\right|^{2^*-2}\hva &\hbox{ in } \frac{\Om-\xl}{\sqrt{\ma}}, \\
				\hva= 0  &\hbox{ on } \partial\left( \frac{\Om-\xl}{\sqrt{\ma}}\right),
			\end{array}\right.
		\end{equation*}
where we have let $\hha(x)=h(\xl+\sa x)$. By \eqref{controlva} and standard elliptic theory there thus exists $\hat{v}_{\infty}\in C^2(\rr^n\backslash\{0\})$ such that $ \hva \to \hat{v}_{\infty}$ in $C_{\loc}^2(\rr^n\backslash\{0\})$, and Theorem \ref{maintheo1} shows that for any $x \in \R^n \backslash \{0\}$ we have 
		\begin{equation*} 
			\hat{v}_{\infty}(x) = (n(n-2))^{\frac{n-2}{2}}|x|^{2-n} \pm v_{\infty}(\xin).
				\end{equation*}
The change of variables $x=\xl+\sa y$ and straightforward computations then show that 
	\begin{equation} \label{pohoint1}
	\begin{aligned}
		&\ma^{-\frac{n-2}{2}}\int_{\partial\bsa(\xl)} F_{\alpha}(x)\, d\sigma(x)\\
		&=\int_{\partial B_{\delta}(0)} \l x, \nu\r\left(  \frac{|\nabla \hva|^2}{2}+\ma\hha \frac{\hva^2}{2}-\ma\frac{|\hva|^{2^*}}{2^*}\right)d\sigma(x) \\
		& -\int_{\partial B_{\delta}(0)}\left( \l x, \nabla \hva \r+\frac{n-2}{2}\hva \right)\partial_{\nu}\hva  d\sigma(x) \\
		& = \pm \frac{\omega_{n-1}}{2}n^{\frac{n-2}{2}}(n-2)^{\frac{n+2}{2}}\vo(\xin) + \ve_\delta + o(1)
		\end{aligned}
	\end{equation}
as $\alpha \to + \infty$, where $\ve_\delta$ denotes a quantity such that $\lim_{\delta \to 0} \ve_\delta = 0$ and where $\omega_{n-1}$ is the area of the round sphere $\mathbb{S}^{n-1}$. We now claim that the following holds:
	\begin{equation}\label{eq:estpart1}
	\begin{aligned}
		 \int_{\bsa(\xl)}& \left( \ha(x)+\frac12 \l\nabla \ha(x),x-\xl \r \right)\va^2\, dx \\=	
		& \left\{\begin{array}{ll}
		\vspace{0.1cm}	O\left( \ma^{\frac{3}{2}}\right)  &\hbox{ if } n=3\\
		\vspace{0.1cm}	O\left( \ma^2\ln\left(\frac{1}{\ma}\right)\right)	 &\hbox{ if } n=4\\
			\ma^2\left( \hi(\xin)\int_{\rr^n} B_0(x)^2\, dx+o(1)\right)  &\hbox{ if } n\geq 5,
		\end{array}\right.
		\end{aligned} 
	\end{equation}
	where $B_0$ is defined in \eqref{B0}. We prove  \eqref{eq:estpart1}. First, using \eqref{contprojbulle} and Theorem \ref{maintheo1}, straightforward computations show that 
\begin{equation}\label{eq:estpart1:1}
\begin{aligned}
 \int_{\bsa(\xl)}&\frac12 \l\nabla \ha(x),x-\xl \r \va^2\, dx\\
  =&\left\{\begin{array}{ll}
		\vspace{0.1cm}	O\left( \ma^{2}\right)  &\hbox{ if } n=3,4\\
			O \left( \ma^3 |\ln \ma| \right)  &\hbox{ if } n\geq 5,
		\end{array}\right.
\end{aligned} 
\end{equation} 
and that 
	\begin{equation}\label{eq:estpart1:2}
	\begin{aligned}
	 \int_{\bsa(\xl)}&  \ha(x)\va^2\, dx = \left\{\begin{array}{ll}
		\vspace{0.1cm}	O\left( \ma^{\frac{3}{2}}\right)  &\hbox{ if } n=3\\
		\vspace{0.1cm}	O\left( \ma^2\ln\left(\frac{1}{\ma}\right)\right)	 &\hbox{ if } n=4. \\
		\end{array}\right.
		\end{aligned} 
		\end{equation}
 If $n\geq 5$, and using Theorem \ref{maintheo1}, we have 
 $$  \int_{\bsa(\xl)}  \ha(x)\va^2\, dx =  \int_{\bsa(\xl)}  \ha(x)\big( \Pi B_\alpha\big) ^2\, dx + o(\ma^2). $$
Dominated convergence together with  \eqref{contprojbulle2} now shows that 
\begin{equation*}
 \int_{\bsa(\xl)}  \ha(x)\big(\Pi B_\alpha\big)^2\, dx = \hi(\xin)\int_{\rr^n} \ma^2 B_0(x)^2\, dx + o(\ma^2).
\end{equation*} 
Combining the latter with \eqref{eq:estpart1:1} and \eqref{eq:estpart1:2} proves \eqref{eq:estpart1}. Combining \eqref{poa1}, \eqref{pohoint1} and \eqref{eq:estpart1} now shows that
\begin{eqnarray}\label{alldim1}	&&\pm\frac{\omega_{n-1}}{2}n^{\frac{n-2}{2}}(n-2)^{\frac{n+2}{2}}\vo(\xin)\ma^{\frac{n-2}{2}} + \ve_\delta\ma^{\frac{n-2}{2}} +o(\ma^{\frac{n-2}{2}})\nonumber\\
	&&=\left\{\begin{array}{ll}
	\vspace{0.1cm}	O\left( \ma^{\frac{3}{2}}\right)  &\hbox{ if } n=3\\
	\vspace{0.1cm}	O\left( \ma^2\ln\left(\frac{1}{\ma}\right)\right)	 &\hbox{ if } n=4\\
		\ma^2\left( \hi(\xin)\int_{\rr^n} B_0^2\, dx+o(1)\right)  &\hbox{ if } n\geq 5.
	\end{array}\right.
\end{eqnarray}
Assume first that $n\in\{3,4,5\}$. Equation \eqref{alldim1} then gives
\begin{eqnarray*} 
	\vo(\xin)+\ve_\delta + o(1)=	\left\{\begin{array}{ll}
		\vspace{0.1cm}O\left( \ma\right)  &\hbox{ if } n=3\\
	\vspace{0.1cm}	O\left( \ma\ln\left(\frac{1}{\ma}\right)\right)	 &\hbox{ if } n=4\\
		O\left( \sa\right)  &\hbox{ if } n=5,
	\end{array}\right.
\end{eqnarray*}
as $\alpha \to + \infty$. Letting first $\alpha \to + \infty$ then $\delta \to 0$ shows that $v_\infty(x_\infty) = 0$. Since $v_\infty \ge 0$ by \eqref{limvlaL2star} and the assumption $I_{\hi}(\Om)<K_n^{-2}$, the strong maximum principle then shows that $v_\infty \equiv 0$. 

Assume now that $n=6$. Integrating $-\Delta B_0 = B_0^2$ shows that 
\begin{equation*} 
	\int_{\rr^6} B_0^2\, dx=6^24^3 \omega_5.
\end{equation*}
Therefore,  it follows from \eqref{alldim1} that 
\begin{eqnarray*}
\pm \frac{\omega_{5}}{2}6^{2}4^4\vo(\xin)\ma^{2}+\ve_\delta \ma^2 + o(\ma^{2})=6^24^3\omega_5\hi(\xin)\ma^{2}+o(\ma^{2}).
\end{eqnarray*}
Letting $\alpha \to +\infty$ and then $\delta \to 0$ shows that 
\begin{equation*} 
	\hi(\xin)= \pm2\vo(\xin).
\end{equation*}
Assume finally that $n \ge 7$. Then $\ma^{\frac{n-2}{2}}= o(\ma^2)$ as $\alpha \to + \infty$, and equation \eqref{alldim1} then gives, after letting $\alpha \to + \infty$,
$$ h_\infty(x_\infty) = 0. $$
These considerations prove Proposition \ref{prop:blowup:interieur} in the case $n \ge 6$. 

\medskip

To conclude the proof of Proposition \ref{prop:blowup:interieur} we now consider the case where $3 \le n \le 5$ and $v_\infty \equiv 0$. We let $\delta >0$ be small enough so that $B_{\delta}(x_\alpha) \subset \Om$ for all $\alpha$ and we write a Pohozaev identity in $B_{\delta}(x_\alpha)$:
	\begin{eqnarray}\label{poa1:u00}
		&&\int_{B_{\delta}(x_\alpha)}\left(\ha(x)+\frac12 \l\nabla \ha(x),x-\xl \r\right)  \va^2\, dx =\int_{B_{\delta}(x_\alpha)}  F_{\alpha}(x) \,  d\sigma(x),
	\end{eqnarray}
where $F_\alpha$ is again as in \eqref{balph}. For $x\in\Om$ we let in this case
		\begin{equation*}
			\hva(x)= \ma^{-\frac{n-2}{2}} \va(x). 
		\end{equation*}
Using \eqref{critvlambda} it is easily seen that $\hva$ satisfies 
		\begin{equation*}
			\left\{\begin{array}{ll}
				-\Delta \hva +h_\alpha \hva=\ma^2 \left|\hva\right|^{2^*-2}\hva &\hbox{ in } \Om, \\
				\hva= 0  &\hbox{ on } \partial \Om,
			\end{array}\right.
		\end{equation*}
and \eqref{contprojbulle} and  \eqref{controlva} show that we have 
$$ |\hva(x)| \le \frac{C}{|x-x_{\alpha}|^{n-2}} \quad \text{ for all } x \in \Om \backslash \{x_\alpha\} $$
where $C$ is a positive constant independent of $\alpha$. Standard elliptic theory with \eqref{limtvaB0} then shows that $ \hva \to \hat{v}_{\infty}$ in $C_{\loc}^2(\overline{\Om} \backslash\{x_\infty\})$, where
$$ \hat{v}_{\infty}(x) = (n-2)\omega_{n-1} \big(n(n-2)\big)^{\frac{n-2}{2}} G_\infty(x_\infty, x)$$
and where $G_\infty$ the Green's function for $-\Delta + \hi$ with Dirichlet boundary conditions in $\Om$, which is the only solution to 
\begin{equation*}	
		\left\{\begin{array}{ll}
			-\Delta_{y} G_{h_\infty}(x,y)+h G_{h_\infty}(x,y)=\delta_{x}&\hbox{ in } \Omega, \\
			 G_{h_\infty}(x,y)= 0  &\hbox{ for } y\in \partial \Om, x\in \Om.
		\end{array}\right.
	\end{equation*}
When $n=3$ it is well-known that we have 
\begin{equation*}
		G_{\infty}(x_\infty,y)=\frac{1}{4\pi |x-y|}+m_{h_\infty}(x_\infty) + O(|x_\infty-y|)  \bb{ for all } y\in \Om\backslash \{x_\infty\}.
	\end{equation*} 
Straightforward computations with the latter then show that 
\begin{equation} \label{poho2:u00}
\begin{aligned}
\ma^{2-n} \int_{B_{\delta}(x_\alpha)}  F_{\alpha}(x) \,  d\sigma(x) = 
\left \{ 
\begin{aligned}
& 24\pi^2 m_{h_\infty}(x_\infty)  + \ve_\delta + o(1) & n=3 \\
& O(1) & n = 4,5 ,
\end{aligned}
\right. 
\end{aligned}
\end{equation}
where $\lim_{\delta \to 0} \ve_\delta =0$. Independently, straightforward computations using Theorem \ref{maintheo1} (see e.g. \cite[Section 5]{PremoselliRobert}) show that 
\begin{equation}\label{poho3:u00}
	\begin{aligned}
		 \int_{B_{\delta}(\xl)}& \left( \ha(x)+\frac12 \l\nabla \ha(x),x-\xl \r \right)\va^2\, dx \\=	
		& \left\{\begin{array}{ll}
		\vspace{0.1cm}	O\left( \delta \ma\right)  &\hbox{ if } n=3\\
		\vspace{0.1cm}	64 \omega_3 h_\infty(x_\infty)\ma^2\ln\left(\frac{1}{\ma}\right) + O\big( \ma^2 \big)	 &\hbox{ if } n=4\\
			\ma^2\left( \hi(\xin)\int_{\rr^n} B_0(x)^2\, dx+o(1)\right)  &\hbox{ if } n\geq 5
		\end{array}\right.
		\end{aligned} 
	\end{equation}
	as $\alpha \to + \infty$. If $n \in \{4,5\}$, combining  \eqref{poho2:u00} and \eqref{poho3:u00} in \eqref{poa1:u00} shows that 
	$$ h_\infty(x_\infty) + o(1) = \left \{ \begin{aligned} & O\left(\ln \left( \frac{1}{\ma} \right)^{-1} \right) & n=4 \\ & O (\ma) & n = 5 \end{aligned} \right. $$
	as $\alpha \to + \infty$, which shows that $h_\infty(x_\infty) = 0$. If $n=3$, combining  \eqref{poho2:u00} and \eqref{poho3:u00} in \eqref{poa1:u00} shows that 
	$$ m_{h_\infty}(x_\infty) + o(1) + \ve_\delta = O(\delta) $$
	as $\alpha \to + \infty$. Letting first $\alpha \to + \infty$ then $\delta \to 0$ proves that $m_{h_\infty}(x_\infty) = 0$, which concludes the proof of Proposition \ref{prop:blowup:interieur}. 
\end{proof}

\subsection{boundary blow-up case: $x_\infty \in \partial \Om$} We assume in this subsection that $x_\infty \in \partial\Om$. For $\alpha \ge 1$, we let
\begin{equation} \label{def:da}
 d_\alpha = d\big(x_\alpha, \partial \Om \big) \to 0 
 \end{equation}
as $\alpha \to +\infty$, since $x_\infty \in \partial \Om$. We know from \eqref{dasurma} that $d_\alpha >> \ma$ as $\alpha \to + \infty$.  For $\alpha \ge 1$ we also let 
\begin{equation}\label{ra}
r_{\la}=\frac{\sqrt{\ma}}{d_{\la}^{\frac{1}{n-2}}},
\end{equation}
and we analyse the bubbling behavior of $v_\alpha$ at the scale $r_\alpha$.  The idea to consider the scale $r_\alpha$ comes from the following heuristic. Recall that when $v_\infty >0$, Hopf's lemma shows that 
$$v_\infty \big(x_\infty - t \nu(x_\infty) \big) = \big( -\partial_{\nu} v_\infty(x_\infty) \big) t + o(t)$$ 
as $t\to 0$. At distance $d_\alpha$ from $\partial \Om$, $v_\infty$ thus behaves at first-order as $\big( -\partial_{\nu} v_\infty(x_\infty) \big) d_\alpha $. The scale $r_\alpha$ thus defines the distance from $x_\alpha$ at which $B_\alpha$ and $v_\infty$ become of the same size. We analyse the boundary blow-up of $v_\alpha$ according to the value of $\frac{d_\alpha}{r_\alpha}$. We first prove the following result, that states that boundary blow-up points cannot get too close from $\partial \Om$:

\begin{prop} \label{prop:blowup:bord}
Let $\Om$ be a smooth bounded domain of $\rr^n$, $n\geq 3$. Let $(\ha)_{\alpha\in\nn}$ be a sequence of functions that converges in $C^{1}(\overline{\Om})$ to $\hi$, where $-\Delta+\hi$ is coercive in $H^1_0(\Om)$ and where $I_{\hi}(\Om)<K_{n}^{-2}$, and we let $(\va)_{\alpha\in\nn}\in H_0^1(\Om)$ be a sequence of solutions of \eqref{critvlambda} that satisfies \eqref{limvlaL2star}, \eqref{limvlaLinfini} and \eqref{sumvla}. Let $x_\infty = \lim_{\alpha \to + \infty} x_\alpha$ and assume that $x_\infty \in \partial \Om$. If $n \ge 6$, assume in addition that $\hi \neq 0$ in $\overline{\Omega}$. Then, up to a subsequence,
$$ \frac{d_{\alpha}}{r_\alpha}  \to + \infty  $$
as $\alpha \to + \infty$. 
\end{prop}

\begin{proof} We proceed by contradiction and we assume that, up to a subsequence, 
\begin{equation}\label{daoverra1}
			\lim_{\la\to\ls} \frac{d_{\la}}{r_{\la}}=\rho\in[0,+\infty).
		\end{equation}
		In this case we define, for all $x\in \frac{\Om-\xl}{d_{\la}}$,
		\begin{equation}\label{vb2}
			\vb(x):=\frac{d_{\la}^{n-2}}{\ma^{\frac{n-2}{2}}}v_{\la}(\xl+d_{\la}x). 
		\end{equation}
Equation \eqref{critvlambda} and the definition of $\vb$ show that $\vb$ satisfies
\begin{equation}\label{critvlambdaz}
	\left\{\begin{array}{ll}
		-\Delta\vb -d_{\la}^2\hb \vb=\left(\frac{\ma}{d_{\la}} \right)^2 \left|\vb\right|^{2^*-2}\vb &\hbox{ in } \frac{
			\Om-\xl}{d_{\alpha}}, \\
		\vb= 0  &\hbox{ on } \partial\left(  \frac{
			\Om-\xl}{d_{\alpha}}\right),
	\end{array}\right.
\end{equation}
where $\vb$ as in \eqref{vb2} and $\hb(x):=h(\xl+\da x)$.
By \eqref{ra}  and \eqref{daoverra1} we have 
\begin{equation} \label{condda}
d_{\alpha}=O\Big( \ma^{\frac{n-2}{2(n-1)}}\Big) \quad \text{ or, equivalently, } \quad  \frac{d_\alpha^{n-2}}{\ma^{\frac{n-2}{2}}} \cdot d_\alpha = O(1). 
\end{equation}
By Hopf's lemma we have 
\begin{equation} \label{hopflemma1}
	\vo(\xl+d_{\la}x)= \vo(\xl) + O(d_\alpha) = O(d_\alpha) 
	\end{equation}
as $\alpha \to +\infty$, and the latter remains obviously true if $v_\infty \equiv 0$. The latter with \eqref{contprojbulle} and Theorem \ref{maintheo1} show that 
\begin{equation} \label{eq:contbarva}
	\big|\vb(x) \big| \leq C\left( 1+|x|^{2-n}\right)\bb{ for all } x\in \frac{
		\Om-\xl}{d_{\alpha}} \backslash \{0\}
\end{equation}	
for some positive constant $C$. Since $\Om$ is smooth and since $d_\alpha \to 0$ as $\alpha \to + \infty$ by assumption, standard elliptic theory shows that, up to a rotation,  $\vb \to \bar{v}_{\infty}\in C^2(\overline{\Om_0} \backslash \{0\})$, where we have let 
\begin{equation}\label{defOm0}
		\Om_0:= ]-\infty,1[\times \rr^{n-1} \bb{ as }\alpha\to +\infty
\end{equation}
and where $\bar{v}_{\infty}$ satisfies 
\begin{eqnarray}\label{critvbinfty}
		-\Delta\bar{v}_\infty =0 &\hbox{ in } \Om_{0}\backslash\{0\} \bb{ , }	\bar{v}_\infty= 0  &\hbox{ on } \partial \Om_{0},
\end{eqnarray}
and
\begin{equation}\label{barvbounded}
	\big|\bar{v}_\infty (x) \big| \leq C\left( 1+|x|^{2-n}\right)\bb{ for all } x\in \Om_0.
\end{equation}
\begin{lemme} \label{lem:CVbord}
We have 
	\begin{equation} \label{lem:CVbord1}
		\bar{v}_\infty(x)=\frac{(n(n-2))^{\frac{n-2}{2}}}{|x|^{n-2}}+\h(x) \bb{ for all }x\in \Om_0\backslash\{0\},
	\end{equation}
where $\h$ satisfies 
	\begin{eqnarray}\label{critvlambdax}
			-\Delta \h=0 &\hbox{ in } \Om_{0} \bb{ , } \hspace{0.1cm}	\h= -(n(n-2))^{-\frac{n-2}{2}}|\cdot|^{2-n}  &\hbox{ on } \partial \Om_{0},
	\end{eqnarray}
	and $\h(0) < 0$. 
\end{lemme}

\begin{proof}[Proof of Lemma \ref{lem:CVbord}]
 Let $ 0 < \delta <1$ be fixed and let $x\in \partial B_{\delta}(0) \backslash \{0\}$. For $\alpha \ge 1$.  Lemma \ref{lem:controleprojbulle} in the Appendix shows that the following holds true: 
\begin{equation} \label{eq:improvedasympto}
\frac{\da^{n-2}}{\ma^{\frac{n-2}{2}}} \Pi B_\alpha \big( x_\alpha+ \da x \big) = \frac{(n(n-2))^{\frac{n-2}{2}}}{|x|^{n-2}} + o(1) + \frac{\ve(|x|)}{|x|^{n-2}}, 
\end{equation}
as $\alpha \to + \infty$, where $\ve(|x|)$ denotes a function that satisfies $\lim_{|x| \to 0} \ve(|x|) = 0$. We now consider $\bar{v}_\infty$ satisfying \eqref{critvbinfty}. By \eqref{barvbounded} and B\^ocher's theorem \cite{axler1992bocher, bocher} there exist $\Lambda \neq 0$ and a harmonic function $\h$ in $\Omega_0$ such that 
\begin{equation}
	\vbi(x)=\Lambda |x|^{2-n}+\h(x) \bb{ for } x\in \Omega_0. 
\end{equation} 
Theorem \ref{maintheo1} together with \eqref{condda} shows that 
$$ \big| \bar{v}_\alpha(x) -  \frac{\da^{n-2}}{\ma^{\frac{n-2}{2}}} \Pi B_\alpha \big( x_\alpha+ \da x \big) \big| \le C + o(1) $$
for $x \in B_\delta(0) \backslash \{0\}$, for some fixed $C >0$ as $\alpha \to + \infty$. Multiplying the latter by $|x|^{n-2}$ and passing to the limit as $\alpha \to + \infty$ then shows, using \eqref{eq:improvedasympto}, that 
$$ \big| |x|^{n-2}\bar{v}_\infty(x) -  \big( 1 + \ve(|x|)\big) \big(n(n-2)\big)^{\frac{n-2}{2}}\big| \le C|x|^{n-2}. $$
Letting $x \to 0$ then shows that $\Lambda =  \big(n(n-2)\big)^{\frac{n-2}{2}}$ and proves \eqref{lem:CVbord1}. That $\h$ satisfies  \eqref{critvlambdax} is a simple consequence of the Dirichlet boundary conditions. 

\medskip

To conclude the proof of Lemma \ref{lem:CVbord} we thus need to show that $\h(0) <0$. For $x \in \Omega_0$ as in \eqref{defOm0} we define 
	\begin{equation} \label{defth}
	\th(x)=	2\frac{n^{\frac{n-4}{2}}(n-2)^{\frac{n-2}{2}}}{\omega_{n-1}}(x_1-1)\int_{\partial\Om_0}|y|^{2-n}|x-y|^{-n} \, d\sigma(y).
	\end{equation}
If $y \in \Omega_0$ we let $y^{*}:=(2-y_1,y^{\prime})\in \rr^n$ be its symmetric with respect to the hyperplane $\{y_1=1\}$. For $x,y \in \Omega_0$, $x \neq y$, we let 
$$G_0(x,y)=\frac{1}{(n-2)\omega_{n-1}}\left( |x-y|^{2-n}-|x-y^{*}|^{2-n}\right) $$ 
be the Green's function of the $-\Delta$ in $\Om_0$ with Dirichlet boundary conditions. Straightforward computations show that 
\begin{equation*}
	\partial_{\nu} G_0(x,y)=\frac{2(x_1-1)}{n w_{n-1}}\frac{1}{|x-y|^{n}} \bb{ for }  x\in \Om_0, \bb{ and } y\in \partial \Om_0,
\end{equation*}
so that $\th$ rewrites as 
\begin{equation*}
	\th(x)=\int_{\partial\Om_0}\frac{(n(n-2))^{\frac{n-2}{2}}}{|y|^{n-2}}\partial_{\nu}  G_0(x,y) \, d\sigma(y).
\end{equation*}
In particular, $\th$ satisfies
$$-\Delta \th=0 \hbox{ in } \Om_{0} \bb{ , } \hspace{0.1cm}	\th= -(n(n-2))^{-\frac{n-2}{2}}|\cdot|^{2-n}  \hbox{ on } \partial \Om_{0} $$
and we  have 
\begin{equation}
\label{tho}
	\th(0)=-2\frac{(n(n-2))^{\frac{n-2}{2}}}{nw_{n-1}}\int_{\rr^{n-1}} \left( 1+|y^{\prime}|^2\right)^{1-n}\,dy^{\prime} < 0. 
	\end{equation}
We now claim that 
\begin{equation} \label{eq:hegaleh}
\h = \th \quad \text{ in } \Omega_0.
\end{equation}
To prove \eqref{eq:hegaleh} we first prove that $\th \in L^\infty(\Omega_0)$. We write any $y \in \partial \Omega_0$ as $y = (1, y')$ with $y' \in \R^n$. We similarly write $x \in \Omega_0$ as $x = (x_1, x')$ with $x_1 < 1$. If $x \in \Omega_0$, with \eqref{defth} and a simple change of variables we thus have, for some positive constant $C = C(n)$,
\begin{equation*}
\begin{aligned}
	|\th(x)| & \le C(1-x_1) \int_{\partial \Omega_0} \frac{1}{\big((x_1-1)^2 + |y'|^2\big)^{\frac{n}{2}}} \, dy' \\
	& \le C \int_{\partial \Omega_0}\frac{1}{\big(1 + |y'|^2\big)^{\frac{n}{2}}}  \, dy' < + \infty,
	\end{aligned} 
\end{equation*}
where the last line again follows from a change of variables. Thus $\th$ is bounded in $\Omega_0 \backslash B_{\ve_0}(1)$. We can now conclude the proof of Lemma \ref{lem:CVbord}. Since $\h$ is harmonic in $\Omega_0$ it is bounded in $B_{\frac12}(0)$. Equations \eqref{barvbounded} and \eqref{lem:CVbord1} also show that $\h$ is bounded in $\Omega_0$. Independently, we just proved that $\th \in L^\infty(\Omega_0)$. The function $\h - \th$ is thus harmonic in $\Omega_0$, bounded in $\Omega_0$ and vanishes on $\partial \Omega_0$. Since $\partial \Omega_0$ is a hyperplane a simple reflection argument allows to apply Liouville's theorem, which shows that $\h \equiv \th$. This proves \eqref{eq:hegaleh} and by \eqref{tho} conclude the proof of Lemma \ref{lem:CVbord}. 
\end{proof}

We are now in position to prove Proposition \ref{prop:blowup:bord}. Let $\delta >0$ be fixed. We write Pohozaev's identity \eqref{poa} in $U_\alpha = B_{\delta d_\alpha}(x_\alpha)$: this gives 
\begin{eqnarray}\label{poa3}
	&&\int_{B_{\delta\da}(\xl)}\left(\ha(x)+\frac{\l\nabla \ha(x),x-\xl \r}{2} \right)  \va^2\, dx =\int_{\partial B_{\delta\da}(\xl)}  F_{\alpha}(x) \,  d\sigma(x),
\end{eqnarray}
where $F_{\alpha}$ is defined in \eqref{balph}. Changing variables we get that 
\begin{eqnarray}\label{Falpha}
	&&\left( \frac{\ma}{\da}\right)^{2-n} \int_{\partial B_{\delta\da}(\xl)}  F_{\alpha}(x) \,  d\sigma(x)\nonumber\\
	&=&\int_{\partial B_{\delta}(0)}\l x, \nu\r\left(  \frac{|\nabla \vb|^2}{2}+\hb \da^2 \frac{\vb^2}{2}-\da^2\frac{|\vb|^{2^*}}{2^*}\right)d\sigma(x)\\
	&&-\int_{\partial B_{\delta}(0)}\left( \l x, \nabla \vb \r+\frac{n-2}{2}\vb \right)\partial_{\nu}\vb \, d\sigma(x),\nonumber
\end{eqnarray}
where $\vb$ is defined in \eqref{vb2}. 
Direct calculations using \eqref{condda} and \eqref{eq:contbarva} yield, since $\ha\in L^{\infty}(\Om)$,
\begin{equation} \label{convrest}
\begin{aligned}
	& \da^2	\int_{\partial B_{\delta}(0)} \l x, \nu\r\hb\vb^2\, d\sigma(x)=O\left(\da^2\delta^{4-n}+\ma^{\frac{n-2}{n-1}}\delta^n\right) = o(1) \quad \text{ and } \\
	& \da^2\int_{\partial B_{\delta}(0)} \l x, \nu\r|\va|^{2^*}\, d\sigma(x)=O\left(\delta^{-n}\da^2+\ma^{\frac{n-2}{n-1}}\delta^n \right) = o(1)
	\end{aligned}
\end{equation}
as $\alpha \to + \infty$. Plugging \eqref{convrest} in  \eqref{Falpha} gives, since $\vb \to \bar{v}_{\infty}\in C^2(\overline{\Om_0} \backslash \{0\})$,
\begin{equation}\label{partfa}
\begin{aligned}
&	\lal \left( \frac{\ma}{\da}\right)^{2-n} \int_{\partial B_{\delta\da}(\xl)}  F_{\alpha}(x) \,  d\sigma(x)\\
	&=\int_{\partial B_{\delta}(0)}|x|\left(  \frac{|\nabla \vbi|^2}{2}-(\partial_{\nu}\vbi)^2\right)\, d\sigma(x)-\frac{n-2}{2}\int_{\partial B_{\delta}(0)} \vbi\partial_{\nu}\vbi\, d\sigma(x)\\
	&=\frac{\omega_{n-1}}{2} n^{\frac{n-2}{2}} (n-2)^{\frac{n+2}{2}} \h(0)+\ve(\delta),
\end{aligned}
\end{equation} 
where $\ve(\delta)\to 0$ as $\delta \to 0$ and where the last equality follows from Lemma \ref{lem:CVbord}. Independently, direct computations using \eqref{haconvhi}, \eqref{limtvaB0} and  \eqref{controlva} show that 
\begin{equation} \label{partha}
\begin{aligned}
&	\int_{B_{\delta\da}(\xl)}\left(\ha(x)+\frac{\l\nabla \ha(x),x-\xl \r}{2} \right)  \va^2\, dx\\
&=	\left\{\begin{array}{ll}
	\vspace{0.1cm}O\left(\delta^3 \da^5+\delta \ma \da\right) &\hbox{ if } n=3\\
	\vspace{0.1cm}	O\left(\delta^4 \da^6+\ma^2\ln\left( \frac{\da}{\ma}\right) \right) &\hbox{ if } n=4\\
\ma^2 \hi(\xin) \int_{\rr^n} B_0(x)^2\, dx +o(\ma^2) + O \big(\delta^n d_\alpha^{n+2}\big)  &\hbox{ if } n\geq 5.
\end{array}\right.
\end{aligned}
\end{equation}
Combining \eqref{partfa} and  \eqref{partha} into \eqref{poa3} we finally obtain that 
\begin{equation} \label{eq:pohobord1}
\begin{aligned}
&\frac{\omega_{n-1}}{2} n^{\frac{n-2}{2}} (n-2)^{\frac{n+2}{2}}\h(0)+\ep(\delta) =	\left( \frac{\da}{\ma}\right)^{n-2} \\
& \times \left\{\begin{array}{ll}
	\vspace{0.1cm}O\left(\delta^3 \da^5+\delta \ma \da\right) &\hbox{ if } n=3\\
	\vspace{0.1cm}	O\left(\delta^4 \da^6+\ma^2\ln\left( \frac{\da}{\ma}\right) \right) &\hbox{ if } n=4\\
\ma^2 \hi(\xin) \int_{\rr^n} B_0(x)^2\, dx +o(\ma^2) + O \big(\delta^n d_\alpha^{n+2}\big)  &\hbox{ if } n\geq 5.
\end{array}\right.
\end{aligned}
\end{equation}
Using \eqref{condda}, and since $d_\alpha \to 0$, we easily obtain that, when $n \in \{3,4,5\}$, \eqref{eq:pohobord1} shows that 
$$ \h(0)+\ep(\delta) = o(1) $$
as $\alpha \to + \infty$, which is a contradiction with Lemma \ref{lem:CVbord}. If now $n \ge 6$, \eqref{condda} shows that $ d_{\alpha}^{n+2} = o (\ma^2)$. Since $\h(0) <0$ by Lemma \ref{lem:CVbord}, we can choose $\delta$ fixed but small enough so that $\h(0) + \ve(\delta) <0$. By \eqref{eq:pohobord1} we then have 
$$ \hi(\xin) \int_{\rr^n} B_0(x)^2\, dx + o(1) \le 0. $$
Letting $\alpha \to + \infty$ implies that $ \hi(\xin) \le 0$. In the case where $h_\infty >0$ in $\overline{\Om}$ this is a contradiction and concludes the proof of Proposition \ref{prop:blowup:bord}. 

We may thus assume that $h_\infty < 0$ in $\overline{\Om}$ and $n \ge 6$. With \eqref{eq:pohobord1} we obtain 
\begin{equation} \label{eq:pohobord12}
d_\alpha = \big( C_0 + o(1) \big) \ma^{\frac{n-4}{n-2}}
\end{equation}
for some constant $C_0 >0$ that depend on $n$ and $h_\infty$. Integrating \eqref{critvlambda} against $\nabla v_\alpha$ in $U_\alpha$ yields the following Pohozaev identity:
\begin{equation} \label{pohogradient:1}
\int_{\partial U_\alpha} \Big( \frac12 |\nabla v_\alpha|^2 \nu - \partial_{\nu} v_\alpha \nabla v_\alpha - \frac{1}{2^*} v_{\alpha}^{2^*} \nu \Big) d \sigma  = - \frac12 \int_{U_\alpha} h_\alpha \nabla( v_\alpha^2) dx,
\end{equation}
where $\nu$ is the outer unit normal to $U_\alpha$. Straightforward computations using Theorem \ref{maintheo1}, \eqref{contprojbulle} and \eqref{hopflemma1} show that 
$$ \int_{\partial U_\alpha}  \frac{1}{2^*} v_{\alpha}^{2^*} \nu  d \sigma=  O \big( \ma^n \da^{-n-1}\big) + O(d_\alpha^{n+1}),$$
while  integrating by parts and using Theorem \ref{maintheo1} and \eqref{contprojbulle} shows that 
$$ \begin{aligned}
\int_{U_\alpha} h_\alpha \nabla( v_\alpha^2) dx &= \int_{\partial U_\alpha} h_\alpha v_\alpha^2 \nu d \sigma - \int_{U_\alpha} v_\alpha^2 \nabla h_\alpha dx \\
& = O \big( \ma^{n-2} \da^{3-n}\big) + O(d_\alpha^{n+1}) + O (\ma^2) . 
\end{aligned} $$
Independently, \eqref{barvbounded} and \eqref{lem:CVbord1} show that 
$$ \begin{aligned} \int_{\partial U_\alpha} \Big( \frac12 |\nabla v_\alpha|^2 \nu - \partial_{\nu} v_\alpha \nabla v_\alpha  \Big) d \sigma & = \frac{\ma^{n-2}}{\da^{n-1}} \left(  \int_{\partial B_\delta(0)} \Big( \frac12 |\nabla \bar{v}_\infty|^2 \nu - \partial_{\nu} \bar{v}_\infty \nabla \bar{v}_\infty  \Big) d \sigma  + o(1) \right)  \\
& =  \frac{\ma^{n-2}}{\da^{n-1}} \left(  n^{\frac{n-2}{2}} (n-2)^{\frac{n+2}{2}} \omega_{n-1} \nabla \h(0) + \ve(\delta)  + o(1) \right)
\end{aligned}   $$ 
as $ \alpha \to + \infty$. Plugging these estimates into \eqref{pohogradient:1} finally gives: 
$$ \begin{aligned} \nabla \h(0) + \ve(\delta) & = O \left( \left( \frac{\ma}{\da} \right)^2  + \frac{d_\alpha^{2n} }{\ma^{n-2}} + \da^2 + \frac{d_\alpha^{n-1}}{\ma^{n-4}}\right) 
& = o(1),
\end{aligned} $$ 
where in the last line we used \eqref{eq:pohobord12}. Passing to the limit as $\alpha \to + \infty$ and as $\delta \to 0$ shows that $\nabla \h(0) = 0$. But going back to \eqref{defth}, and since $\h = \th$, we have $\partial_1 \h(0) < 0$ by Lemma \ref{h0negatif} below, which is a contradiction. This concludes the proof of Proposition \ref{prop:blowup:bord}. 
\end{proof}

We now investigate more precisely what happens at the scale $r_\alpha$. This is the content of the following result: 
\begin{prop} \label{prop:blowup:bord2}
Let $\Om$ be a smooth bounded domain of $\rr^n$, $n\geq 3$. Let $(\ha)_{\alpha\in\nn}$ be a sequence of functions that converges in $C^{1}(\overline{\Om})$ to $\hi$, where $-\Delta+\hi$ is coercive in $H^1_0(\Om)$ and where $I_{\hi}(\Om)<K_{n}^{-2}$, and we let $(\va)_{\alpha\in\nn}\in H_0^1(\Om)$ be a sequence of solutions of \eqref{critvlambda} that satisfies \eqref{limvlaL2star}, \eqref{limvlaLinfini} and \eqref{sumvla}. Let $x_\infty = \lim_{\alpha \to + \infty} x_\alpha$ and assume that $x_\infty \in \partial \Om$. Assume that 
$$ \frac{d_{\alpha}}{r_\alpha}  \to + \infty  $$
as $\alpha \to + \infty$. Then 
\begin{itemize}
\item If $n \in \{3,4,5\}$ we have $v_\infty \equiv 0$.
\item If $n \ge6$ we have $h_\infty(x_\infty)= 0$. 
	\end{itemize}
\end{prop}

\begin{proof}
We assume that 
\begin{equation}\label{daoverra}
\lim_{\la\to\ls} \frac{d_{\la}}{r_{\la}}=+\infty.
\end{equation}
Using \eqref{ra} we define, for $x \in \frac{\Om-\xl}{r_{\la}}$, 
\begin{equation}\label{vb1}
\vb(x)=\frac{r_\alpha^{n-2}}{\ma^{\frac{n-2}{2}}}\va(\xl+r_{\la}x) = d_{\la}^{-1}\va(\xl+r_{\la}x). 
\end{equation}
Since $\va$ satisfies \eqref{critvlambda}, $\vb$ solves 
\begin{equation*}
	\left\{\begin{array}{ll}
\vspace{0.1cm}		-\Delta\vb +r_{\la}^2\hb \vb=r_{\la}^2 d_{\la}^{\frac{4}{n-2}} \left|\vb\right|^{2^*-2}\vb &\hbox{ in } \frac{\Om-\xl}{r_{\la}}, \\
		\vb= 0  &\hbox{ on } \partial\left(  \frac{\Om-\xl}{r_{\la}}\right),
	\end{array}\right.
\end{equation*}
where we have let $\hb(x)=h(\xl+\ra x)$. By Hopf's lemma and by \eqref{daoverra} we have 
\begin{equation} \label{asyuo}
	\vo(\xl+r_{\la}x)= \vo(\xl) + O(r_\alpha) = -\partial_{\nu}\vo(\xin)d_{\alpha}+o(d_{\alpha}) 
	\end{equation}
as $\alpha \to +\infty$, and \eqref{asyuo} obviously remains true if $\vo \equiv 0$. Using \eqref{contprojbulle}, Theorem \ref{maintheo1}, \eqref{ra} and \eqref{asyuo} we thus have
\begin{eqnarray*}\label{estva1}
	\big|\vb(x) \big|&\leq	&C \Big( |x|^{2-n}+ 1 \Big) \bb{ for all }x\in \frac{\Om-\xl}{r_{\la}}\backslash\{0\}.
\end{eqnarray*}
Standard elliptic theory then shows that $\vb$ converges to some $\bar{v}_\infty$ in $C^2_{loc}(\rr^n\backslash\{0\})$. Let $x \in \rr^n\backslash\{0\}$ be fixed. First, as a consequence of Lemma \ref{lem:controleprojbulle} in the Appendix, the following holds: 
\begin{equation*} 
\frac{r_\alpha^{n-2}}{\ma^{\frac{n-2}{2}}} \Pi B_\alpha(\xl+r_{\la}x)  \to (n(n-2))^{\frac{n-2}{2}}|x|^{2-n} \quad \text{ in } C^2_{loc} (\rr^n\backslash\{0\}) 
  \end{equation*}
as $\alpha \to + \infty$. The latter with \eqref{asyuo} and Theorem \ref{maintheo1} then shows that 
 \begin{equation}\label{convba}
\bar{v}_\infty= (n(n-2))^{\frac{n-2}{2}}|x|^{2-n} \pm \partial_{\nu}\vo(\xin) 
\end{equation}
holds. For $\alpha$ large enough we let $U_{\alpha}=B_{\ra}(\xl)\subset \Om$ and we apply the Pohozaev Identity \eqref{poa}. We get 
	\begin{eqnarray}\label{poa2}
	&&\int_{B_{\ra}(\xl)}\left(\ha(x)+\frac{\l\nabla \ha(x),x-\xl \r}{2} \right)  \va^2\, dx =\int_{\partial B_{\ra}(\xl)}  F_{\alpha}(x) \,  d\sigma(x),
\end{eqnarray}
where $F_{\alpha}$ is defined in \eqref{balph}. By changing $x$ into $\xl+\da x$, we can write that 
\begin{eqnarray*}
&&\da^{-2}\ra^{2-n}	\int_{\partial B_{\ra}(\xl)}	F_{\alpha}(x)\, d\sigma(x)\\
	&=&	\int_{\partial B_{1}(0)}\l x, \nu\r\left(  \frac{|\nabla \vb|^2}{2}+\hb\ra^2 \frac{\vb^2}{2}-\ra^2\frac{|\vb|^{2^*}}{2^*}\right)d\sigma(x)\\
	&&-\int_{\partial B_{1}(0)}\left( \l x, \nabla \vb \r+\frac{n-2}{2}\vb \right)\partial_{\nu}\vb\, d\sigma(x),
\end{eqnarray*}
where $\vb$ is as in \eqref{vb1}. Direct calculations with \eqref{controlva} and \eqref{asyuo} give
\begin{equation*} 
\begin{aligned}
	& \ra^2	\int_{\partial B_{1}(0)} \l x, \nu\r\hb\vb^2\, d\sigma=O\left(\ra^2\right) \quad \bb{ and } \\
	& \ra^2\int_{\partial B_{1}(0)} \l x, \nu\r|\vb|^{2^*}\, d\sigma=O\left( \ra^2\right).
\end{aligned}
\end{equation*}
Together with \eqref{convba}, the latter then shows that 
\begin{equation}\label{estpart11}
	\lal\da^{-2}\ra^{2-n}	\int_{\partial B_{\ra}(\xl)}	F_{\alpha}(x)\, d\sigma(x)=\pm \frac{\omega_{n-1}}{2}n^{\frac{n-2}{2}}(n-2)^{\frac{n+2}{2}}\partial_{\nu}\vo(\xin).
\end{equation}
Since $\lal \ra\ma^{-1}=+\infty$, direct computations using \eqref{haconvhi}, \eqref{limtvaB0}, \eqref{controlva}, \eqref{ra} and \eqref{asyuo} show that 
\begin{equation}\label{eq:estpart12}
\begin{aligned}
	& \int_{B_{\ra}(\xl)} \left(\ha(x)+\frac{\l\nabla \ha(x),x-\xl \r}{2} \right) \va^2\, dx \\
	& =	\left\{\begin{aligned}
	&O\left( \frac{\ma^{\frac32}}{d_\alpha} \right)  &\hbox{ if } n=3\\
			& O\left( \ma^2\ln\left(\frac{r_\alpha}{\ma}\right) + \ma^2 \right) &\hbox{ if } n=4\\\
		&\ma^2\left( \hi(\xin)\int_{\rr^n} B_0(x)^2\, dx+o(1)\right)  &\hbox{ if } n\geq 5.
	\end{aligned}\right.
\end{aligned}
\end{equation}
Returning now to \eqref{poa2} with  \eqref{estpart11} and \eqref{eq:estpart12}, and since $d_\alpha^2 \ra^{n-2} = d_\alpha \ma^{\frac{n-2}{2}}$ by \eqref{ra}, we have that
\begin{equation}\label{alldim2}
\begin{aligned}
&\pm \frac{\omega_{n-1}}{2}(n-2)^{\frac{n+2}{2}}n^{\frac{n-2}{2}}\partial_{\nu}\vo(\xin)d_\alpha \ma^{\frac{n-2}{2}}+o(d_\alpha \ma^{\frac{n-2}{2}})\\
&=	\left\{\begin{aligned}
	&O\left( \frac{\ma^{\frac32}}{d_\alpha} \right)  &\hbox{ if } n=3\\
			& O\left( \ma^2\ln\left(\frac{r_\alpha}{\ma}\right) \right) &\hbox{ if } n=4\\\
		&\ma^2\left( \hi(\xin)\int_{\rr^n} B_0(x)^2\, dx+o(1)\right)  &\hbox{ if } n\geq 5.
	\end{aligned}\right.
\end{aligned}
\end{equation}
Independently, since $\ra=o(\da)$ by \eqref{daoverra}, and by \eqref{ra}, we get
\begin{equation}\label{sa}
	\sa=o\left( \da^{\frac{n-1}{n-2}}\right) \bb{ as } \alpha\to +\infty.
\end{equation}
Assume first that $n=3$. Then \eqref{alldim2} shows that 
$$ \partial_{\nu}\vo(\xin) + o(1) = O \left( \frac{\ma}{d_\alpha^2}\right) = o(1) $$
by \eqref{sa}. If $n=4$, \eqref{alldim2} shows that
$$ \partial_{\nu}\vo(\xin) + o(1) = O \left(\frac{\ma}{d_\alpha} \ln\left(\frac{r_\alpha}{\ma}\right) \right) = O \left(\ma^{\frac23} \ln\left(\frac{r_\alpha}{\ma}\right) \right) = o(1) $$
by \eqref{sa}. If $n =5$,  \eqref{alldim2} shows that
$$ \partial_{\nu}\vo(\xin) + o(1) = O \left( \frac{\ma^{\frac12}}{d_\alpha}\right) = o(1) $$
again by \eqref{sa}. We thus obtain, when $n \in \{3,4,5\}$, that 
\begin{equation*}
\partial_{\nu}\vo(\xin)=0,
\end{equation*}
which shows that $v_\infty \equiv 0$ by Hopf's lemma. Assume now that $n \ge 6$. Then \eqref{alldim2} shows that
$$ \hi(\xin)\int_{\rr^n} B_0(x)^2\, dx = O \left( d_\alpha \ma^{\frac{n-6}{2}} \right) + o(1) = o(1) $$
since $d_\alpha \to 0$ as $\alpha \to + \infty$. This concludes the proof of Proposition \ref{prop:blowup:bord2}.
\end{proof}

The next result finally shows that, in small dimensions, the concentration point cannot occur on $\partial \Om$:

\begin{prop} \label{prop:blowup:bord3}
Let $\Om$ be a smooth bounded domain of $\rr^n$, $n\geq 3$. Let $(\ha)_{\alpha\in\nn}$ be a sequence of functions that converges in $C^{1}(\overline{\Om})$ to $\hi$, where $-\Delta+\hi$ is coercive in $H^1_0(\Om)$ and where $I_{\hi}(\Om)<K_{n}^{-2}$, and we let $(\va)_{\alpha\in\nn}\in H_0^1(\Om)$ be a sequence of solutions of \eqref{critvlambda} that satisfies \eqref{limvlaL2star}, \eqref{limvlaLinfini} and \eqref{sumvla}. Let $x_\infty = \lim_{\alpha \to + \infty} x_\alpha$. Assume that $n\in\{3,4\}$ or that $n =5$ and $h_\infty \neq 0$ in $\overline{\Om}$. Then  $x_\infty \in \Om$. 
\end{prop}

\begin{proof}
We proceed by contradiction and assume that $x_\infty \in \partial \Om$. Under the assumptions of Proposition \ref{prop:blowup:bord3}, Propositions \ref{prop:blowup:bord} and \ref{prop:blowup:bord2} also apply. They show in particular that  
\begin{equation} \label{hypothese:finale}
\frac{d_\alpha}{r_\alpha} \to + \infty
\end{equation} as $\alpha \to + \infty$ and that $v_\infty \equiv 0$. For $x \in \frac{\Om-\xl}{d_{\la}}$ we define again 
		\begin{equation}\label{vb2bis}
			\vb(x):=\frac{d_{\la}^{n-2}}{\ma^{\frac{n-2}{2}}}v_{\la}(\xl+d_{\la}x). 
		\end{equation}
Equation \eqref{critvlambda} then shows that $\vb$ satisfies
\begin{equation*} 
	\left\{\begin{array}{ll}
		-\Delta\vb -d_{\la}^2\hb \vb=\left(\frac{\ma}{d_{\la}} \right)^2 \left|\vb\right|^{2^*-2}\vb &\hbox{ in } \frac{
			\Om-\xl}{d_{\alpha}}, \\
		\vb= 0  &\hbox{ on } \partial\left(  \frac{
			\Om-\xl}{d_{\alpha}}\right),
	\end{array}\right.
\end{equation*}
where $\hb(x):=h(\xl+\da x)$. Since $v_\infty \equiv 0$, \eqref{contprojbulle} and Theorem \ref{maintheo1} show that 
\begin{equation} \label{eq:contbarva2}
	\big|\vb(x) \big| \leq C |x|^{2-n}\bb{ for all } x\in \frac{
		\Om-\xl}{d_{\alpha}} \backslash \{0\}
\end{equation}	
for some positive constant $C$. Since $\Om$ is smooth and since $d_\alpha \to 0$ as $\alpha \to + \infty$ by assumption, standard elliptic theory shows that, up to a rotation,  $\vb \to \bar{v}_{\infty}\in C^2(\overline{\Om_0} \backslash \{0\})$  as $\alpha\to +\infty$, where $\Om_0:= ]-\infty,1[\times \rr^{n-1}$  and where $\bar{v}_{\infty}$ satisfies 
\begin{eqnarray*} 
		-\Delta\bar{v}_\infty =0 &\hbox{ in } \Om_{0}\backslash\{0\} \bb{ , }	\bar{v}_\infty= 0  &\hbox{ on } \partial \Om_{0}
\end{eqnarray*}
and	
\begin{equation*} 
|\bar{v}_\infty (x) \big| \leq C |x|^{2-n} \quad \text{ for all } x\in \Om_0.
\end{equation*} 
The arguments in the proof of Lemma \ref{lem:CVbord} again show that 
	\begin{equation}  \label{lem:CVbord2}
		\bar{v}_\infty(x)=\frac{(n(n-2))^{\frac{n-2}{2}}}{|x|^{n-2}}+\h(x) \bb{ for all }x\in \Om_0\backslash\{0\},
	\end{equation}
where $\h$ satisfies 
	\begin{eqnarray*} 
			-\Delta \h=0 &\hbox{ in } \Om_{0} \bb{ , } \hspace{0.1cm}	\h= -(n(n-2))^{-\frac{n-2}{2}}|\cdot|^{2-n}  &\hbox{ on } \partial \Om_{0},
	\end{eqnarray*}
is given for any $x \in \Omega$ by 
\begin{equation} \label{expression:h:final}
\h(x) = 2\frac{n^{\frac{n-4}{2}}(n-2)^{\frac{n-2}{2}}}{\omega_{n-1}}(x_1-1)\int_{\partial\Om_0}|y|^{2-n}|x-y|^{-n} \, d\sigma(y),
\end{equation}
and satisfies
\begin{equation} \label{contra:h0}
\h(0) < 0.
\end{equation} 
In the following we let $0< \delta < 1$ and  $U_\alpha = B_{\delta d_\alpha}(x_\alpha)$. We write Pohozaev's identity \eqref{poa} in $U_\alpha$: this gives 
\begin{eqnarray*} 
	&&\int_{B_{\delta\da}(\xl)}\left(\ha(x)+\frac{\l\nabla \ha(x),x-\xl \r}{2} \right)  \va^2\, dx =\int_{\partial B_{\delta\da}(\xl)}  F_{\alpha}(x) \,  d\sigma(x),
\end{eqnarray*}
where $F_{\alpha}$ is defined in \eqref{balph}. Mimicking the computations that led to  \eqref{Falpha}, \eqref{convrest} and \eqref{partfa} we obtain that 
\begin{equation} \label{eq:rajout:1} 
\begin{aligned}
\int_{\partial B_{\delta\da}(\xl)}  &F_{\alpha}(x) \,  d\sigma(x) \\
&=   \left(\frac{\ma}{\delta \da}\right)^{n-2} \left( \frac{\omega_{n-1}}{2} n^{\frac{n-2}{2}} (n-2)^{\frac{n+2}{2}} \h(0)+\ve(\delta) + o(1) \right)
\end{aligned}
\end{equation} 
as $\alpha \to + \infty$, where $\ve(\delta)\to 0$. Independently, direct computations using \eqref{haconvhi}, \eqref{limtvaB0} and  \eqref{controlva} show that 
\begin{equation} \label{eq:rajout:2} 
\begin{aligned}
	& \int_{B_{\ra}(\xl)} \left(\ha(x)+\frac{\l\nabla \ha(x),x-\xl \r}{2} \right) \va^2\, dx \\
	& =	\left\{\begin{aligned}
	&O\left( \ma \ra  \right)  &\hbox{ if } n=3\\
			& 64 \omega_3 h_\infty(x_\infty)  \ma^2\ln\left(\frac{d_\alpha}{\ma}\right) + O( \ma^2 )  &\hbox{ if } n=4\\\
		&\ma^2\left( \hi(\xin)\int_{\rr^n} B_0(x)^2\, dx+o(1)\right)  &\hbox{ if } n\geq 5.
	\end{aligned}\right.
\end{aligned}
\end{equation}
If $n=3$, combining \eqref{eq:rajout:1} and \eqref{eq:rajout:2} gives 
$$ \h(0) = O ( \sqrt{\ma}), $$
hence $\h(0) = 0$, which is a contradiction with \eqref{contra:h0}. This proves Proposition \ref{prop:blowup:bord3} when $n=3$. If $n=4,5$, and using \eqref{contra:h0}, we obtain $h_\infty(x_\infty) \le 0$. If $h_\infty >0$ in $\overline{\Om}$ this is a contradiction and concludes the proof of Proposition \ref{prop:blowup:bord3}. 

We assume from now on that $h_\infty <0$ in $\overline{\Om}$ and $n=4,5$. In this case the proof is similar to the proof of Proposition \ref{prop:blowup:bord} when $n \ge 6$.  Using again  \eqref{contra:h0} the previous Pohozaev's identity then shows the existence of a constant $C_0 >0$ depending on $n, h_\infty$ and $\delta$ such that
\begin{equation} \label{pohogradient:0}
\begin{aligned}
& \da^2 \ln \left(\frac{d_\alpha}{\ma}\right) = C_0 +o(1) & \text{ if } n = 4 \\
& \da = (C_0 + o(1)) \ma^{\frac13} & \text{ if } n = 5.
\end{aligned}
\end{equation}
We recall the gradient Pohozaev identity  \eqref{pohogradient:1}:
\begin{equation*}
\int_{\partial U_\alpha} \Big( \frac12 |\nabla v_\alpha|^2 \nu - \partial_{\nu} v_\alpha \nabla v_\alpha - \frac{1}{2^*} v_{\alpha}^{2^*} \nu \Big) d \sigma  = - \frac12 \int_{U_\alpha} h_\alpha \nabla( v_\alpha^2) dx,
\end{equation*}
where $\nu$ is the outer unit normal to $U_\alpha$. Straightforward computations using Theorem \ref{maintheo1} and \eqref{contprojbulle} show that 
$$ \int_{\partial U_\alpha}  \frac{1}{2^*} v_{\alpha}^{2^*} \nu  d \sigma=  O \big( \ma^n \da^{-n-1}\big),$$
while  integrating by parts and using Theorem \ref{maintheo1} and \eqref{contprojbulle} shows that 
$$ \begin{aligned}
\int_{U_\alpha} h_\alpha \nabla( v_\alpha^2) dx &= \int_{\partial U_\alpha} h_\alpha v_\alpha^2 \nu d \sigma - \int_{U_\alpha} v_\alpha^2 \nabla h_\alpha dx \\
& = O \big( \ma^{n-2} \da^{3-n}\big) + \left \{ \begin{aligned} &  O \left( \ma^2\ln\left(\frac{d_\alpha}{\ma}\right) \right)  & \text{ if } n = 4 \\ & O (\ma^2) & \text{ if } n = 5 \end{aligned} \right \} . 
\end{aligned} $$
Independently, \eqref{eq:contbarva2} and \eqref{lem:CVbord2} show that 
$$ \begin{aligned} \int_{\partial U_\alpha} \Big( \frac12 |\nabla v_\alpha|^2 \nu - \partial_{\nu} v_\alpha \nabla v_\alpha  \Big) d \sigma & = \frac{\ma^{n-2}}{\da^{n-1}} \left(  \int_{\partial B_\delta(0)} \Big( \frac12 |\nabla \bar{v}_\infty|^2 \nu - \partial_{\nu} \bar{v}_\infty \nabla \bar{v}_\infty  \Big) d \sigma  + o(1) \right)  \\
& =  \frac{\ma^{n-2}}{\da^{n-1}} \left(  n^{\frac{n-2}{2}} (n-2)^{\frac{n+2}{2}} \omega_{n-1} \nabla \h(0) + \ve(\delta)  + o(1) \right)
\end{aligned}   $$ 
as $ \alpha \to + \infty$. Plugging these estimates into \eqref{pohogradient:1} finally gives: 
$$ \begin{aligned} \nabla \h(0) + \ve(\delta) & = O \left( \left( \frac{\ma}{\da} \right)^2\right) + O(\da^2) + 
\left \{
\begin{aligned}
& O \left( \da^3 \ln \left(\frac{d_\alpha}{\ma}\right) \right)& \text{ if } n = 4 \\
& O \left( \frac{\da^4}{\ma} \right) &  \text{ if } n = 5
\end{aligned}  \right.   \\
& = o(1),
\end{aligned} $$ 
where in the last line we used \eqref{pohogradient:0}. Passing to the limit as $\alpha \to + \infty$ and as $\delta \to 0$ shows that $\nabla \h(0) = 0$. But going back to \eqref{expression:h:final} we again have $\partial_1 \h(0) < 0$ by Lemma \ref{h0negatif} below, which is a contradiction. This concludes the proof of Proposition \ref{prop:blowup:bord3} when $n=4,5$ and $h_\infty <0$.

To conclude the proof of Proposition \ref{prop:blowup:bord3} we finally assume that $n=4$. If $h_\infty(x_\infty) \neq 0$ in $\overline{\Om}$ the proof of Proposition \ref{prop:blowup:bord3} follows from the previous arguments. We may then assume that $h_\infty(x_\infty) = 0$. In this case combining \eqref{eq:rajout:1} and \eqref{eq:rajout:2} shows that 
$$ \h(0) = O(d_\alpha^2) = o(1) $$
as $\alpha \to + \infty$, which contradicts \eqref{contra:h0}. This concludes the proof of Proposition  \ref{prop:blowup:bord3}.
\end{proof}

\begin{rem} \label{concentration:au:bord}
Assume that $(\va)_{\alpha\in\nn}\in H_0^1(\Om)$ is any sequence of solutions of \eqref{critvlambda} that satisfies \eqref{limvlaL2star} and \eqref{limvlaLinfini}, so that \eqref{sumvla}, \eqref{Bla} and \eqref{limitefaiblevla} also hold. Let $x_\infty = \lim_{\alpha \to  \infty} x_\alpha$ be the concentration point of $u_\alpha$. Propositions \ref{prop:blowup:bord}, \ref{prop:blowup:bord2} and \ref{prop:blowup:bord3} prove that $x_\infty \in \Om$, i.e. that $x_\infty$ is an interior blow-up point, in the following cases (regardless of the value of $\vo$): either when $n \in \{3,4\} $ or when $n \ge 5$ and under the assumption $h_\infty \neq 0$ in $\overline{\Om}$. If $h_\infty$ is allowed to vanish somewhere in $\partial \Om$ the property that $x_\infty \in  \Om$ is unlikely to remain true, and concentration points may arise on the boundary in large dimensions. When $n \ge 7$, for instance, \emph{sign-changing} solutions of \eqref{BNlambda} that blow-up with one concentration point in $\partial \Om$ as $\lambda \to 0_+$ (which corresponds to $h_\infty \equiv 0$) have been constructed in \cite{vaira2015new} (see also \cite{musso2024nodal} for a more recent construction with an arbitrary number of bubbles).
\end{rem}

\begin{rem} \label{concentration:au:bord2}
We mentioned in Remark \ref{concentration:au:bord} that when $n \ge 7$ and $h_\infty \equiv 0$ \emph{sign-changing} solutions of \eqref{BNlambda} that blow-up with one concentration point in $\partial \Om$ as $\lambda \to 0_+$ exist (see \cite{vaira2015new}). By contrast, it is important to point out that, in any dimension $n \ge 4$, \emph{positive} solutions of \eqref{BNlambda} as $\lambda \to 0_+$ may only blow-up with interior concentration points and do not possess concentration points in $\partial \Om$. This is shown in \cite[Proposition 2.1]{LaurainKonig2}, and heavily relies on the positivity of the solutions. The issue of boundary concentration points thus arises when working with \emph{sign-changing} solutions of \eqref{BNalpha}. 
\end{rem}

 We are now in position to prove Theorem \ref{maintheo2bis}.
 
 \begin{proof}[End of the proof of Theorem \ref{maintheo2bis}]
 	Let $\Om$ be a smooth bounded domain of $\rr^n$, $n\geq 3$, and $(\ha)_{\alpha\in\nn}$ be sequence that converges in $C^{1}(\overline{\Om})$ towards $\hi$. Assume that $-\Delta+\hi$ is coercive and that $I_{\hi}(\Om)<K_{n}^{-2}$. Let $(\va)_{\alpha\in\nn}\in H_0^1(\Om)$ be a sequence of solutions of \eqref{critvlambda} that satisfies \eqref{limvlaL2star}. Assume first that $(\va)_{\alpha \in \mathbb{N}}$ is, up to a subsequence, uniformly bounded in $L^\infty(\Om)$. By standard elliptic theory it then strongly converges, again up to a subsequence, to some $v_0$ in $C^2(\overline{\Om})$ as $\alpha \to + \infty$. That $v_0 \neq 0$ simply follows from the coercivity of $-\Delta + h_\infty$ which easily implies, by Sobolev's inequality, that $\liminf_{\alpha \to + \infty} \Vert \va \Vert_{H^1_0} >0$. This concludes the proof of Theorem \ref{maintheo2bis}.
	
We thus proceed by contradiction and assume that, up to a subsequence, \eqref{limvlaLinfini} holds, and hence that \eqref{sumvla}, \eqref{Bla} and \eqref{limitefaiblevla} hold for some sequence $(x_{\alpha})_{\alpha \in \mathbb{N}}$ of points in $\Om$ and $(\ma)_{\alpha \in \mathbb{N}}$ of positive real number converging to $0$ satisfying \eqref{defma}. In particular,
$$ v_\alpha = B_{\alpha} \pm v_\infty + o(1) \quad \text{ in } H^1_0(\Om),$$
where $v_{\infty}\equiv 0$ or $v_\infty$ is a positive solution of \eqref{critu0}. We let $x_\infty = \lim_{\alpha \to + \infty} x_\alpha \in \overline{\Om}$. Under these assumptions, the analysis of Section \ref{secpohozaev} applies. 

We first assume that  $n \ge 7$ and that $\hi \neq 0$ at every point of $\overline{\Om}$. Propositions  \ref{prop:blowup:bord} and  \ref{prop:blowup:bord2}  first show that $x_\infty \in \Om$. Proposition \ref{prop:blowup:interieur} then applies and shows that $\hi(x_\infty) = 0$, which is a contradiction.

We now assume that $3 \le n \le 5$ and that $(\va)_{\alpha\in\nn}\in H_0^1(\Om)$ is \emph{sign-changing} for all $\alpha \ge 0$. We then claim that we have 
\begin{equation} \label{contra:petites:dims}
 v_\infty >0 \quad \text{ in } \Om. 
 \end{equation}
This is a strong consequence of the assumption that $\va$ changes sign. We adapt an argument from \cite[Lemma 3.1]{CeramiSoliminiStruwe}. Since $\va$ does not strongly converge to $\vo$, $(\va)_+$ and $(\va)_-$ may not simultaneously strongly converge to $(\vo)_+$ and $(\vo)_-$. Assume for simplicity that $(\va)_+$ weakly but not strongly converges to $(\vo)_+$ in $H^1_0(\Om)$. Integrating \eqref{critvlambda} against $(\va)_+$ and using Br\'ezis-Lieb lemma shows that 
$$ \int_\Om |\nabla ( (\va)_+ - (\vo)_+) |^2\, dx + o(1) = \int_{\Om} |(\va)_+ - (\vo)_+|^{2^*} \, dx, $$
from which we deduce that $ \int_{\Om} |(\va)_+ - (\vo)_+|^{2^*} \, dx \ge K_n^{-n} + o(1)$ as $\alpha \to + \infty$ by \eqref{defkn}. Independently, since $(\va)_-$ is nonzero, integrating \eqref{critvlambda} against $(\va)_-$ and using \eqref{Iho} yields $ \int_{\Om} |(\va)_-|^{2^*} \, dx \ge I_{h_\alpha}(\Om)^{\frac{n}{2}}$. 
Thus, again by Br\'ezis-Lieb's lemma,
$$ \begin{aligned} 
\int_\Om |\va|^{2^*} \, dx & = \int_{\Om} |(\va)_+ |^{2^*} \, dx + \int_{\Om} |(\va)_- |^{2^*} \, dx   \\
& =  \int_{\Om} |(\va)_+ - (\vo)_+|^{2^*} \, dx +  \int_{\Om} |(\vo)_+|^{2^*} \, dx +  \int_{\Om} |(\va)_-|^{2^*} \, dx + o(1) \\
& \ge K_n^{- n} + I_{h_\infty}(\Om)^{\frac{n}{2}} + o(1)
\end{aligned} $$ 
as $\alpha \to + \infty$. This shows that $\vo \not \equiv 0$ and hence that $\vo >0$ in $\Om$ and attains $I_{h_\infty}(\Om)$. As before, the analysis of Section \ref{secpohozaev} applies to $\va$. First, using \eqref{contra:petites:dims}, Propositions \ref{prop:blowup:bord} and  \ref{prop:blowup:bord2} show that $x_\infty \in \Om$. We may thus apply Proposition \ref{prop:blowup:interieur}, which shows that $\vo \equiv 0$ and contradicts \eqref{contra:petites:dims}. Thus $(\va)_{\alpha \in \mathbb{N}}$ is again uniformly bounded in $L^\infty(\Om)$ and Theorem \ref{maintheo2bis} is proven. 
 \end{proof}
 
 We now prove Corollary \ref{maincorol}:
 
 \begin{proof}[Proof of Corollary \ref{maincorol}]
We assume that $\Om$ and $h$ are as in the assumptions of Corollary \ref{maincorol}. We observed in the proof of Theorem \ref{maintheo2bis} that any sequence $(\va)_{\alpha \in \mathbb{N}}$ of solutions of \eqref{BN} which is bounded in $L^\infty(\Om)$ up to a subsequence is precompact in $C^2(\overline{\Om})$. With this observation we proceed by contradiction: if no $\ve$ as in the statement of Corollary \ref{maincorol} exists, we can find a sequence $(\va)_{\alpha \in \mathbb{N}}$ of solutions of 
 \begin{equation*}
 	\left\{\begin{aligned}
		-\Delta \va +h \va &=\left|\va\right|^{2^*-2}\va \hbox{ in } \Omega \\
		\va&= 0  \hbox{ in } \partial \Omega
	\end{aligned}\right.
\end{equation*}
which satisfies $\lim_{\alpha \to + \infty} \Vert \va \Vert_{\infty} = + \infty$ and $\limsup_{\alpha \to + \infty} \int_{\Om} |\va|^{2^*} \, dx \le K_n^{-n} + I_h(\Om)^{\frac{n}{2}}$. When $3 \le n \le 5$ we have in addition that $(\va)_{\alpha \in \mathbb{N}}$ changes sign. We may now apply Theorem \ref{maintheo2bis} to the sequence $(\va)_{\alpha \in \mathbb{N}}$ with $h_\alpha \equiv h$ for all $\alpha \ge 0$, which gives a contradiction. 
 \end{proof}
 
 We now consider the six-dimensional case and prove Proposition \ref{prop:cas6}: 
 
 \begin{proof}[Proof of Proposition \ref{prop:cas6}]
Assume indeed that $(v_\alpha)_{\alpha \in \mathbb{N}}$ is a sequence of solutions of \eqref{critvlambda} that satisfies \eqref{limvlaL2star} and \eqref{limvlaLinfini}. Hence \eqref{sumvla}, \eqref{Bla} and \eqref{limitefaiblevla} hold for some sequence $(x_{\alpha})_{\alpha}$ of points in $\Om$ and $(\ma)_{\alpha}$ of positive real number converging to $0$ satisfying \eqref{defma}. Then
$$ v_\alpha = B_{\alpha} \pm v_\infty + o(1) \quad \text{ in } H^1_0(\Om),$$
where $v_{\infty}\equiv 0$ or $v_\infty$ is a positive solution of \eqref{critu0}. We let $x_\infty = \lim_{\alpha \to + \infty} x_\alpha \in \overline{\Om}$. First, Propositions  \ref{prop:blowup:bord} and \ref{prop:blowup:bord2} show that  $x_\infty \in \Om$. Proposition \ref{prop:blowup:interieur} then applies and shows that $h_\infty(x_\infty) = \pm 2 v_\infty(x_\infty)$. 
 \end{proof}
 
 \begin{rem}
 When $n \in \{3,4,5\}$ Theorem \ref{maintheo2bis} is likely to be false if \eqref{cond:energy1} is not satisfied. On a closed Riemannian manifold, and when $3 \le n \le 5$, blowing-up solutions of equations like \eqref{BNalpha} of the form $B_\alpha + v_\infty$, where $v_\infty$ is a \emph{sign-changing} solution of \eqref{BN}, may exist: see  \cite[Theorem 1.4]{PremoselliVetois2}. The examples in \cite[Theorem 1.4]{PremoselliVetois2} are constructed on a closed manifold with symmetries and $B_\alpha$ concentrates at a point where $v_\infty$ vanishes. These examples are likely to adapt to the case of a symmetric bounded open set when $3 \le n \le 5$ and $h_\infty \neq 0$ in $\overline{\Om}$. They suggest that, even when $3 \le n \le 5$, sign-changing solutions may exhibit non-compactness at a higher energy level than $K_{n}^{-n}+I_{\hi}(\Om)^{\frac{n}{2}}$.  \end{rem}
 
 \appendix

\section{Technical results} \label{annexe}

\subsection{Pointwise estimates on $\Pi B_\alpha$}

Let $\Pi B_\alpha$ be given by \eqref{projbulle}. We prove a technical result that was used several times through the paper and that provides an asymptotic expansion of $\Pi B_\alpha$ close to $\partial \Om$: 
\begin{lemme} \label{lem:controleprojbulle}
Let $(\xl)_{\alpha \in \nn}$ and $(\ma)_{\alpha \in \nn}$ be respectively sequences of points in $\Om$ and positive real numbers, satisfying $d(x_\alpha, \partial \Om) >> \ma$ as $\alpha \to + \infty$. Let $B_\alpha$ be given by \eqref{Bla} and $\Pi B_\alpha$ be given by  \eqref{projbulle}. Let $(y_\alpha)_{\alpha \in \mathbb{N}}$ be a sequence of points in $\Omega$ satisfying 
\begin{equation} \label{lem:appendix1}
d(y_\alpha, \partial \Om) \to 0, \quad  |x_\alpha - y_\alpha| \le \frac12 d(x_\alpha, \partial \Om) \quad \text{ and } \quad \frac{|x_\alpha - y_\alpha|}{\ma} \to + \infty 
\end{equation}
as $\alpha \to + \infty$. Let $\ell = \lim_{\alpha \to + \infty} \frac{ |x_\alpha - y_\alpha|}{d(x_\alpha, \partial \Om)}$ which exists up to a subsequence. Then, as $\alpha \to + \infty$, we have
$$ \Pi B_\alpha(y_\alpha) = \Big( \big( n(n-2) \big)^{\frac{n-2}{2}} +o(1) + \ve(\ell) \Big) \frac{\ma^{\frac{n-2}{2}} }{|x_\alpha - y_\alpha|^{n-2}} $$
where $\ve: \R_+ \to \R_+$ denotes a function such that $\ve(0) = 0$ and $\lim_{x \to 0} \ve(x) = 0$. 
\end{lemme}

\begin{proof}
We write a representation formula for $\Pi B_\alpha$ using \eqref{projbulle}: 
\begin{equation} \label{eq:improvedasympto0}
 \Pi B_\alpha(y_\alpha) = \int_{\Om} G_\alpha(y_\alpha, \cdot) B_\alpha^{2^*-1} dx 
 \end{equation} 
 where as before $G_\alpha$ denotes the Green's function of $- \Delta + h_\alpha$ with Dirichlet boundary conditions in $\Omega$. Using \eqref{lem:appendix1}, \eqref{estGlai} and arguing as in \eqref{theorieC01} we have
\begin{equation} \label{eq:improvedasympto2}
\int_{\Om \backslash B_{\frac{|x_\alpha - y_\alpha|}{2}}(x_\alpha)} G_\alpha(y_\alpha, \cdot)B_\alpha^{2^*-1} dx= o \big( B_\alpha(y_\alpha) \big)  
\end{equation}
 as $\alpha \to +\infty$. We let in what follows
\begin{equation*} 
I_{\alpha}:= |x_\alpha - y_\alpha|^{n-2}\ma^{-\frac{n-2}{2}}\int_{B_{\frac{|x_\alpha - y_\alpha|}{2}}(x_{\alpha})} G_{\la}(y_\alpha,\cdot)B_\alpha^{2^*-1} \, dx. 
\end{equation*}
By a change of variable we have 
\begin{equation}\label{A1alpha}
\begin{aligned}
I_{\alpha} & =\int_{B_{\frac{|x_\alpha - y_\alpha|}{2 \ma }}(0)} |x_\alpha - y_\alpha|^{n-2} G_{\la}(y_\alpha, x_\alpha + \ma z) B_0(z)^{2^*-1} \, dz \\
\end{aligned}
\end{equation}
where $B_0$ is as in \eqref{B0}. Using \eqref{estGlai} there is $C >0$ such that, for any $z \in  B_{\frac{|x_\alpha - y_\alpha|}{2 \ma }}(0)$, 
$$|x_\alpha - y_\alpha|^{n-2} G_{\la}(y_\alpha, x_\alpha + \ma z) \le C  $$
holds. Let $z \in \R^n $ be fixed. Since $\ma=o(\da)$ we have by \eqref{lem:appendix1}
 $$\begin{aligned}
 D:=	\lal   \frac{d(y_\alpha,\partial \Om)d(x_\alpha + \ma z,\partial \Om)}{\big| y_\alpha - (x_\alpha + \ma z) \big|^2 }  & \ge \frac{1}{ \ell^2}(1 - \ell)  
   \end{aligned} $$ 
as $\alpha \to + \infty$, where we have let $\ell = \lim_{\alpha \to + \infty} \frac{ |x_\alpha - y_\alpha|}{d(x_\alpha, \partial \Om)}$ and with the convention that the right-hand side is equal to $+ \infty$ if $\ell = 0$. Note that $\ell \le \frac12$ by \eqref{lem:appendix1}. Since $\ma = o(\da)$ and $\lal |y_\alpha - (x_\alpha + \ma z)|=0$ uniformly in $z \in B_R(0)$, Proposition $12$ in \cite{RobDirichlet} applies and shows that for any fixed $z \in \R^n$, 
\begin{equation} \label{eq:improvedasympto3}
\begin{aligned}
	  \lal |x_\alpha - y_\alpha|^{n-2}G_{\la}(y_\alpha, x_\alpha + \ma z) 
& =\frac{1}{(n-2)\omega_{n-1}} \Big( 1 - \frac{1}{(1+4D)^{\frac{n-2}{2}}} \Big) \\
& = \frac{1}{(n-2)\omega_{n-1}}\big( 1 + O(\ell) \big).
\end{aligned}
\end{equation}
Plugging \eqref{eq:improvedasympto3} in \eqref{A1alpha} we get by dominated convergence that 
$$ \begin{aligned}
 I_{\alpha} & = \big( 1 + \ve(\ell) + o(1) \big)  \frac{1}{(n-2) \omega_{n-1}} \int_{\R^n} B_0^{2^*-1} \, dx  \\
 & = \big( 1 + \ve(\ell) + o(1) \big) \big(n(n-2)\big)^{\frac{n-2}{2}}
 \end{aligned} $$
 as $\alpha \to + \infty$, where $\ve(\ell)$ denotes a function such that $\ve(0) = 0$ and $\ve(\ell) \to 0$ as $\ell \to 0$. In the latter estimate we used that 
 $\int_{\R^n}B_0^{2^*-1} \, dx = (n-2) \omega_{n-1} \big( n(n-2)\big)^{\frac{n-2}{2}}$ which follows from integrating the equation $- \Delta B_0 = B_0^{2^*-1}$. Going back to the definition of $I_\alpha$ proves the lemma. 
  \end{proof}
  
  \subsection{Sign of $\partial_1 \h(0)$}
  
  We finally prove the following simple result that was used in the proof of Propositions \ref{prop:blowup:bord} and \ref{prop:blowup:bord3}:
  
  \begin{lemme} \label{h0negatif}
Let $\th$ be given by \eqref{defth}. Then $\partial_1 \th(0) < 0$.
  \end{lemme}

\begin{proof}
Straightforward computations show that 
$$ \begin{aligned}
\frac{1}{D_0}\partial_1 \th(0) =  \int_{\partial \Om_0} |y|^{2-2n} d \sigma(y) - n \int_{\partial \Om_0} |y|^{-2n} d \sigma(y) ,
\end{aligned} $$
where we have let $D_0 = 2\frac{n^{\frac{n-4}{2}}(n-2)^{\frac{n-2}{2}}}{\omega_{n-1}}$ and where $\partial \Om_0 = \{1\} \times \R^{n-1}$. Simple changes of variable then yield
$$\begin{aligned}   
 \int_{\partial \Om_0} |y|^{2-2n} d \sigma(y) & = \frac{\omega_{n-2}}{2} I_{n-1}^{\frac{n-3}{2}} \quad \text{ and } \\
 \int_{\partial \Om_0} |y|^{-2n} d \sigma(y)& = \frac{\omega_{n-2}}{2} I_{n}^{\frac{n-3}{2}}
\end{aligned} $$ 
where $\omega_{n-2}$ is the area of the round sphere $\mathbb{S}^{n-2}$ and where we have let, for $p,q >0$, $p > q+1$, 
$$ I_p^q = \int_0^{+ \infty}\frac{r^q}{(1+r)^p} \, dr. $$
Classical induction formulae (see e.g. \cite{AubinYamabe}) show that $I_{n}^{\frac{n-3}{2}} = \frac12 I_{n-1}^{\frac{n-3}{2}}$. Combining these computations finally shows that 
$$\frac{1}{D_0}\partial_1 \th(0)  = \frac{\omega_{n-2}}{2} I_{n-1}^{\frac{n-3}{2}} \Big( 1 - \frac{n}{2} \Big) =  - \frac{n-2}{2}   \int_{\partial \Om_0} |y|^{2-2n} d \sigma(y)  < 0$$
which proves the Lemma.
\end{proof}

\bibliographystyle{amsplain}
\bibliography{biblio}

\end{document}